\documentclass[12pt,twoside,reqno]{amsart}
\usepackage{amsmath}
\usepackage{amsfonts}
\usepackage{amssymb}
\usepackage{color}
\usepackage{mathrsfs}
\usepackage{cite}
\usepackage{cleveref}
\usepackage{geometry}
\allowdisplaybreaks
\newcommand{\RR}{\mathbb R}

\newcommand{\mb}{\mathcal{B}}

\definecolor{cadmiumgreen}{rgb}{0.0, 0.42, 0.24}

\newcommand{\cma}[1]{{\color{magenta}{#1}}}

\newtheorem{theorem}{Theorem}[section]
\newtheorem{corollary}[theorem]{Corollary}
\newtheorem{lemma}[theorem]{Lemma}
\newtheorem{proposition}[theorem]{Proposition}
\newtheorem{definition}[theorem]{Definition}
\newtheorem{remark}[theorem]{Remark}
\numberwithin{equation}{section}
\begin{document}
\title{Singular limit of a chemotaxis model with indirect signal production and phenotype switching} 
\thanks{Corresponding author: Ph.~Lauren\c{c}ot (\texttt{philippe.laurencot@univ-smb.fr})}
\thanks{orcid: 0000-0003-3091-8085 (PhL), 0000-0001-8694-1407 (CS)}
\author{Philippe Lauren\c{c}ot}
\address{Laboratoire de Math\'ematiques (LAMA), UMR~5127, Universit\'e Savoie Mont Blanc, CNRS, F--73000 Chamb\'ery, France}
\email{philippe.laurencot@univ-smb.fr}
\author{Christian Stinner}
\address{Technische Universit\"at Darmstadt, Fachbereich Mathematik, Schlossgartenstr.~7, D--64289 Darmstadt, Germany}
\email{stinner@mathematik.tu-darmstadt.de}
\date{\today}
%
\begin{abstract}
Convergence of solutions to a partially diffusive chemotaxis system with indirect signal production and  phenotype switching is shown in a two-dimensional setting when the switching rate increases to infinity, thereby providing a rigorous justification of formal computations performed in the literature. The expected limit system being the classical parabolic-parabolic Keller-Segel system, the obtained convergence is restricted to a finite time interval for general initial conditions but valid for arbitrary bounded time intervals when the mass of the initial condition is appropriately small. Furthermore, if the solution to the limit system blows up in finite time, then neither of the two phenotypes in the partially diffusive system can be uniformly bounded with respect to the $L_2$-norm beyond that time.
\end{abstract}

\maketitle

%
%
\pagestyle{myheadings}
\markboth{\sc{Ph.~Lauren\c cot \& C.~Stinner}}{Singular limit of a chemotaxis model}

\section{Introduction}

The singular limit $\gamma\to\infty$ of a chemotaxis model with indirect signal production and phenotype switching
\begin{equation}\label{fds}
\begin{array}{ll}
\partial_t u_\gamma = \mathrm{div}(D_u\nabla u_\gamma - u_\gamma \nabla w_\gamma) + \gamma \big( \theta v_\gamma -u_\gamma \big) & \text{ in }\;\; (0,\infty) \times \Omega,  \\[2mm]
\partial_t v_\gamma = D_v \Delta v_\gamma + \gamma \big( u_\gamma - \theta v_\gamma \big) & \text{ in }\;\; (0,\infty) \times \Omega, \\[2mm]
\partial_t w_\gamma = D \Delta w_\gamma - \alpha w_\gamma + v_\gamma & \text{ in }\;\; (0,\infty) \times \Omega,
\end{array}
\end{equation}
supplemented with no-flux boundary conditions and initial conditions $(u_\gamma,v_\gamma,w_\gamma)(0) = (u_0,v_0,w_0)$, is studied in \cite{PaWi2023} in a bounded domain $\Omega \subset \mathbb{R}^n$ with $n \le 3$ and positive constants $D_u$, $D_v$,
$D$, $\alpha$, and $\theta$, which were all set to $1$ in \cite{PaWi2023}. 
The above model describes the dynamics of a population of cells divided into two phenotypes with respective densities $u_\gamma$ and $v_\gamma$ and a chemoattractant with concentration $w_\gamma$. The model is designed to account for phenotype-dependent responses of cells to the presence of a chemoattractant \cite{MLP2022}. While both phenotypes are diffusing with possibly different diffusion rates $D_u$ and $D_v$, only the motion of the phenotype with density $u_\gamma$ features a chemotactic bias, the chemoattractant being only secreted by the other phenotype. For more details concerning the modeling we refer to \cite{MLP2022, PaWi2023} and the references therein. 

As pointed out in \cite[Section~2]{PaWi2023}, letting $\gamma\to\infty$ formally, one finds $u_\gamma \sim \theta v_\gamma$, so that, introducing the total density $n_\gamma := u_\gamma + v_\gamma$, we expect that $u_\gamma\sim \theta n_\gamma/(1+\theta)$ and $v_\gamma\sim n_\gamma/(1+\theta)$. Since $n_\gamma$ solves
\begin{equation*}
	\partial_t n_\gamma = \mathrm{div}(D_u \nabla u_\gamma - u_\gamma \nabla w_\gamma) +  D_v \Delta v_\gamma \qquad\text{ in }\;\; (0,\infty) \times \Omega, 
\end{equation*}
cluster points of $(n_\gamma,w_\gamma)_{\gamma}$ as $\gamma\to\infty$ (if any) are bound to solve the following parabolic-parabolic Keller-Segel system
\begin{align*}
	\partial_t n & = \frac{D_u \theta + D_v}{1+\theta} \Delta n -\frac{\theta}{1+\theta} \mathrm{div} ( n \nabla w) &\text{ in }\;\; (0,\infty) \times \Omega, \\
	\partial_t w & = D \Delta w - \alpha w + \frac{n}{1+\theta} & \text{ in }\;\; (0,\infty) \times \Omega,  
\end{align*}
supplemented with no-flux boundary conditions and initial conditions $(n,w)(0) = (n_0,w_0)$ with $n_0= u_0+v_0$. A challenging issue in establishing the connection between these two systems is that, while the former is globally well-posed in space dimension 2 and 3 \cite{FuSe2017,PaWi2023}, finite time blowup may occur for the latter \cite{HeVe1997, Wi2013}. 

As the diffusion rates of the two phenotypes may be of different orders, we focus here on the limit case where the phenotype which produces the chemoattractant is not diffusing and static ($D_v=0$) and set $D_u=1$ for simplicity. 
For the corresponding system, global well-posedness and the existence of a critical mass in a bounded domain of $\mathbb{R}^2$ are established in \cite{LaSt2021}. It is a particular case of the model from \cite{RLKB2014} for foraging ants and also is a variant of chemotaxis models for building termites nests, mountain pine beetles, and cancer cells. For more details we refer to the introduction of \cite{LaSt2021} and the references therein. 

We therefore assume that $\Omega \subset \mathbb{R}^2$ is a bounded domain with smooth boundary and investigate the singular limit $\gamma\to\infty$ of the system
\begin{subequations}\label{aPL}
\begin{align}
\partial_t u_\gamma & = \mathrm{div}(\nabla u_\gamma - u_\gamma \nabla w_\gamma) + \gamma \big( \theta v_\gamma -u_\gamma \big) &\text{ in }\;\; (0,\infty) \times \Omega, \label{a1PL} \\
\partial_t v_\gamma & = \gamma \big( u_\gamma - \theta v_\gamma \big) & \text{ in }\;\; (0,\infty) \times \Omega, \label{a2PL} \\
\partial_t w_\gamma & = D \Delta w_\gamma - \alpha w_\gamma + v_\gamma & \text{ in }\;\; (0,\infty) \times \Omega, \label{a3PL}
\end{align}
supplemented with no-flux boundary conditions
\begin{equation}
\big( \nabla u_\gamma - u_\gamma \nabla w_\gamma \big) \cdot \mathbf{n} = \nabla w_\gamma \cdot \mathbf{n} =0 \qquad\text{ on }\;\; (0,\infty) \times \partial \Omega \label{bPL}
\end{equation}
\end{subequations}
and initial conditions
\begin{equation}\label{iPL}
  (u_\gamma,v_\gamma,w_\gamma)(0) = (u_0,v_0,w_0) \qquad\text{ in }\;\; \Omega,
\end{equation}
where the parameters $D$, $\alpha$, and $\theta$ are positive.

According to the above formal calculation, we expect that in the limit $\gamma\to\infty$ solutions to~\eqref{aPL}--\eqref{iPL} converge to that of
\begin{subequations}\label{aKS}
\begin{align}
	\partial_t n & = \frac{\theta}{1+\theta} \mathrm{div}(\nabla n - n \nabla w) &\text{ in }\;\; (0,\infty) \times \Omega, \label{a1KS}\\
	\partial_t w & = D \Delta w - \alpha w + \frac{n}{1+\theta} & \text{ in }\;\; (0,\infty) \times \Omega, \label{a2KS} 
\end{align}
supplemented with no-flux boundary conditions
\begin{equation}
	\big( \nabla n - n \nabla w \big) \cdot \mathbf{n} = \nabla w \cdot \mathbf{n} =0 \qquad\text{ on }\;\; (0,\infty) \times \partial \Omega, \label{a3KS}
\end{equation}
\end{subequations}
and initial conditions
\begin{equation}\label{iKS}
	(n,w)(0) = (n_0,w_0) \qquad\text{ in }\;\; \Omega,
\end{equation}
with $n_0= u_0+v_0$, and the purpose of this work is to provide a proof of this fact. Such a result supplements nicely the analysis performed in \cite[Lemmas~4.9 and~4.10]{PaWi2023} for the fully diffusive system~\eqref{fds}, where only a conditional convergence result is provided. As already mentioned, the first difficulty to be faced is that, while solutions to~\eqref{aPL}--\eqref{iPL} are always global in time, see \cite[Theorem~1.1]{LaSt2021} and \Cref{prop:wp} below, solutions to~\eqref{aKS}--\eqref{iKS} may blow up in finite time when $\|n_0\|_1=\|u_0+v_0\|_1$ exceeds the threshold value $4\pi (1+\theta)D$, see \cite{Ho2001, Ho2002, HoWa2001} and \Cref{sec:resc} below. This feature excludes convergence on arbitrary time intervals and we shall thus prove only convergence on a small time interval for arbitrary initial data. It is actually proved in \cite[Theorem~1.4]{PaWi2023} that there are initial conditions such that, for the fully diffusive system \eqref{fds}, the family $(u_\gamma,v_\gamma)_{\gamma>0}$ is not bounded in $L_\infty((0,T)\times\Omega)$ for some $T>0$ sufficiently large and we provide a corresponding result for \eqref{aPL}--\eqref{iPL} in \Cref{thm:ub} below. When $\|n_0\|_1<4\pi (1+\theta)D$, convergence on any finite time interval is established. The second difficulty is of a more technical nature and is due to the assumption that one of the species does not diffuse in space \cite{RLKB2014}. As a consequence, no compactness with respect to the space variable is available for $(v_\gamma)_{\gamma>0}$.

\begin{theorem}\label{thm:sl}
	Assume that $\Omega \subset \mathbb{R}^2$ is a bounded domain with smooth boundary. For $M>0$, we consider
	\begin{equation}
		(u_0,v_0,w_0) \in W_{3,+}^1(\Omega) \times W_{3,+}^1(\Omega) \times W_{2,+}^2(\Omega) \label{eq:ic}
	\end{equation}
	satisfying 
	\begin{equation}
		\|u_0+v_0\|_{1} = M. \label{mass0}
	\end{equation}
	For $\gamma>0$, we denote the solution to \eqref{aPL}--\eqref{iPL}
	given by \Cref{prop:wp} by $(u_\gamma,v_\gamma,w_\gamma)$ and set $n_\gamma := u_\gamma + v_\gamma$. Then there is $T_\infty\in (0,\infty]$ such that, for all $T\in (0,T_\infty)$,
	\begin{subequations}\label{cv}
		\begin{align}
			\lim_{\gamma\to\infty} \int_0^T \|n_\gamma(t) - n(t)\|_2^2\ \mathrm{d}t = 0, \label{cvn} \\
			\lim_{\gamma\to \infty} \sup_{t\in [0,T]} \|w_\gamma(t) - w(t)\|_{W_2^1} = 0, \label{cvw}
		\end{align}
	\end{subequations}
	where $(n,w)$ is the unique weak-strong solution to~\eqref{aKS}--\eqref{iKS} with $n_0 := u_0+v_0$ in the sense of \Cref{def:c1}.
	
	Moreover, $T_\infty=\infty$ if $M = \| u_0+v_0 \|_{1} \in (0,4\pi(1+\theta)D)$.
\end{theorem}

The final statement of \Cref{thm:sl} is valid for all masses below the critical mass and this property can be established thanks to the construction of a Liapunov functional \cite{LaSt2021}. The latter does not seem to be available for the fully diffusive system \eqref{fds} and it is yet unclear whether \Cref{thm:sl} extends to that case.
 
We next establish the analogue of \cite[Theorem~1.4]{PaWi2023} for the partially diffusive system~\eqref{aPL}--\eqref{iPL}.
	
\begin{theorem}\label{thm:ub}
	Consider $(u_0,v_0,w_0)$ satisfying~\eqref{eq:ic} and set $n_0:= u_0+w_0$. Assume that the corresponding classical solution $(n,w)$ to~\eqref{aKS}--\eqref{iKS} blows up at some finite time $T_{bu}\in (0,\infty)$; that is, 
	\begin{equation}
		\limsup_{t\to T_{bu}} \|n(t)\|_\infty = \infty. \label{eq:i1}
	\end{equation}
	Then, for all $T>T_{bu}$,
  \begin{equation}
		\sup_{\gamma\ge 1}\big\{ \|u_\gamma\|_{L_\infty((0,T),L_2(\Omega))}\big\} = \infty \qquad\mbox{and}\qquad
		\sup_{\gamma\ge 1}\big\{ \|v_\gamma\|_{L_\infty((0,T),L_2(\Omega))}\big\} = \infty,
	\end{equation}
	where $(u_\gamma,v_\gamma,w_\gamma)$ denotes the solution to \eqref{aPL}--\eqref{iPL} with initial condition $(u_0,v_0,w_0)$	given by \Cref{prop:wp}.
\end{theorem}

Before we start to prove the above results, we recall in \Cref{sec:wp} the global well-posedness for \eqref{aPL}--\eqref{iPL} which is established in \cite{LaSt2021}. 
The proof of \Cref{thm:sl} involves two steps: the derivation of estimates on $(u_\gamma,v_\gamma,w_\gamma)_\gamma$ which do not depend on $\gamma$, eventually leading to compactness, and afterwards the convergence of $(u_\gamma,v_\gamma,w_\gamma)_\gamma$. The proofs of the former in \Cref{sec:ue} bear some similarities with computations performed to establish the local or global existence of solutions to the Keller-Segel system~\eqref{aKS}--\eqref{iKS} and are divided into small time estimates for arbitrary masses $M>0$ and global estimates for $M \in (0,4\pi(1+\theta)D)$. The convergence proof is more involved and we shall deal with it in \Cref{sec:cv}. One important step in the proof is the strong convergence of $(v_\gamma)_\gamma$ which does not follow from classical compact embeddings due to the non-diffusive character of \eqref{a2PL}. The proof of \Cref{thm:sl} is completed in \Cref{sec:pth}. The final \Cref{sec:ub} is devoted to the proof of \Cref{thm:ub}, where we show that uniform bounds on the $L_2$-norm of either $v_\gamma$ or $u_\gamma$ imply the validity of the estimates which are required for the convergence proof of \Cref{thm:sl} and contradict the assumed blowup in view of the uniqueness of the weak-strong solution to~\eqref{aKS}--\eqref{iKS}. We collect in three appendices some results used in the proofs. In \Cref{sec:mti} we recall the well-known Moser-Trudinger inequality, while the rescaling performed in \Cref{sec:resc} shows the optimality of the mass constraint from \Cref{thm:sl}. In \Cref{sec:uws} we give the definition of the weak-strong solution to \eqref{aKS}--\eqref{iKS} and prove its uniqueness, adapting an argument from \cite{HoWi2005}.

\section{Well-posedness}\label{sec:wp}

In this section we recall the global well-posedness for \eqref{aPL}--\eqref{iPL} established in \cite[Theorem~1.1]{LaSt2021} in an appropriate functional setting and give the notation of the involved function spaces. More precisely, for $r\in (1,\infty)$, we set
\begin{align}\label{in1}
& W^m_{r, \mb}(\Omega) := \left\{ z \in W^m_r (\Omega) \: : \nabla z \cdot \mathbf{n} =0 \mbox{ on } \partial\Omega \right\} \quad\text{if}\quad 1 + \frac{1}{r} < m \le 2, \nonumber \\
& W^m_{r, \mb}(\Omega) := W^m_r (\Omega) \quad\text{if}\quad -1 + \frac{1}{r} < m < 1 + \frac{1}{r}, \\
& W^m_{r, \mb}(\Omega) := W^{-m}_{r/(r-1)} (\Omega)^\prime \quad\text{if}\quad -2 + \frac{1}{r} < m \le -1 + \frac{1}{r}, \nonumber 
\end{align}
and
\begin{equation}\label{in2}
W^m_{r, \mb, +}(\Omega) := \left\{ z \in W^m_{r, \mb} (\Omega) \: : z \ge 0 \mbox{ in } \Omega \right\}, 
\end{equation}
where $W_r^m(\Omega)$, $m\in [0,\infty)$, $r\in [1,\infty)$, denote the usual Sobolev spaces, see \cite[Section~5]{Ama1993}. 

\begin{proposition}\label{prop:wp}
  Let $M>0$, $\gamma>0$, and $(u_0,v_0,w_0) \in W_{3,+}^1(\Omega;\mathbb{R}^3)$ satisfying~\eqref{mass0}.
  	Then the system \eqref{aPL}--\eqref{iPL} has a unique non-negative weak solution $(u_\gamma,v_\gamma,w_\gamma)$ in $W^1_3(\Omega)$ defined on $[0,\infty)$ satisfying
  \begin{align*}
    & u_\gamma \in C \left([0,\infty),W^{1}_{3,+}(\Omega) \right) \cap C^1 \left( [0,\infty), W^1_{3/2} (\Omega , \RR^2)^\prime \right), \\
    & v_\gamma \in C^1([0,\infty),W^1_{3,+}(\Omega)), \\
    & w_\gamma \in C \left([0,\infty),W^{1}_{3,+}(\Omega) \right) \cap C^1 \left( [0,\infty), W^1_{3/2} (\Omega , \RR^2)^\prime \right),
  \end{align*}
  and
  \begin{equation}\label{mass}
    \| (u_\gamma+v_\gamma)(t) \|_{1} = M, \qquad t \ge 0.
  \end{equation}
  Moreover,
  \begin{align*}
  & u_\gamma \in C \left((0,\infty),W^{2}_{3, \mb}(\Omega) \right) \cap C^1 \left( (0,\infty), L_3 (\Omega) \right) \, , \\
  & w_\gamma \in C \left((0,\infty),W^{2}_{3, \mb}(\Omega) \right) \cap C^1 \left( (0,\infty), L_3 (\Omega) \right) \, .
  \end{align*}
\end{proposition}

The local well-posedness of~\eqref{aPL}--\eqref{iPL} relies on the abstract theory developed in \cite{Ama1991a, Ama1991b} for partially diffusive parabolic systems. It is however likely that a larger set of initial data may be handled as well, in light of the recent result in \cite[Section~5.2]{MaWa2023}. Global existence is a consequence from several estimates and makes in particular use of the following Gagliardo-Nirenberg inequality
\begin{equation}
	\|z\|_4 \le c_0 \|z\|_{W_2^1}^{1/2} \|z\|_2^{1/2}, \qquad z\in W_2^1(\Omega), \label{eq:GN}
\end{equation} 
which is also an important tool in the forthcoming analysis.

\section{Convergence}\label{sec:cv}

We begin with the convergence issue and assume that we have already obtained several bounds on the family $(u_\gamma,v_\gamma,w_\gamma)_{\gamma\ge 1}$, which do not depend on $\gamma\ge 1$ and are valid on a time interval $[0,T]$. We shall devote the next sections to the derivation of these bounds.

\begin{proposition}\label{prop:bd}
	Let $M>0$ and consider $(u_0,v_0,w_0) \in W_{3,+}^1(\Omega;\mathbb{R}^3)$ satisfying~\eqref{mass0}. For $\gamma>0$, we denote the solution to \eqref{aPL}--\eqref{iPL} given by \Cref{prop:wp} by $(u_\gamma,v_\gamma,w_\gamma)$ and set 
	\begin{equation}
		n_\gamma := u_\gamma + v_\gamma. \label{eq:ng}
	\end{equation} 
	Assume that there are $T \in (0,\infty)$ and $c_1>0$ such that, for any $\gamma\ge 1$, 
\begin{subequations}\label{ubd}
	\begin{align}
		& \sup_{t\in [0,T]} \big\{ \|u_\gamma(t)\|_2 + \|v_\gamma(t)\|_2  + \|w_\gamma(t)\|_{W_2^1} \big\} \le c_1, \label{ubd1} \\
		& \int_0^T \big[ \|\nabla u_\gamma(t)\|_2^2 + \|\partial_t w_\gamma(t)\|_2^2 + \gamma \|e_\gamma(t)\|_2^2  \big]\ \mathrm{d}t \le c_1^2, \label{ubd2}
	\end{align}
\end{subequations}
with
\begin{equation}
	e_\gamma := \theta v_\gamma - u_\gamma. \label{eq:eg}
\end{equation}
Then there are a sequence ($\gamma_j)_{j\ge 1}$, $\gamma_j\to\infty$, and non-negative functions
\begin{align*}
	& n\in W_2^1((0,T),W_2^1(\Omega)') \cap L_\infty((0,T),L_2(\Omega))\cap L_2((0,T),W_2^1(\Omega)), \\
	& w\in W_2^1((0,T),L_2(\Omega))\cap L_\infty((0,T),W_2^1(\Omega))\cap  L_2((0,T),W_2^2(\Omega)),
\end{align*}
such that
\begin{align}
	& \lim_{j\to\infty} \sup_{t\in [0,T]} \big\{ \| (n_{\gamma_j} -n)(t)\|_{(W_2^1)'} + \|(w_{\gamma_j} - w)(t) \|_{2} \big\} = 0, \label{eq:cvs}\\
	& \lim_{j\to\infty} \int_0^T \big[ \|(w_{\gamma_j}- w)(t)\|_{W_2^1}^2 + \|(n_{\gamma_j}-n)(t)\|_2^2 + \|e_{\gamma_j}(t)\|_2^2 \big]\ \mathrm{d}t = 0. \label{eq:cvi1}
\end{align}
In addition,
\begin{equation}
	\lim_{j\to\infty} \int_0^T \left[ \left\| \left( u_{\gamma_j} - \frac{\theta n}{1+ \theta} \right)(t) \right\|_2^2 + \left\| \left( v_{\gamma_j} - \frac{n}{1+ \theta} \right)(t) \right\|_2^2 \right]\ \mathrm{d}t = 0 \label{eq:cvi2}
\end{equation}
and
\begin{equation}
	\lim_{j\to\infty} \int_0^T \left\| \left( u_{\gamma_j}\nabla w_{\gamma_j} - \frac{\theta}{1+\theta} n \nabla w \right)(t) \right\|_{2}^{4/3}\ \mathrm{d}t = 0. \label{eq:cvi3}
\end{equation}
\end{proposition}

\begin{remark}\label{rem:re}
	At first glance, the estimate on $(w_\gamma)_{\gamma\ge 1}$ in $L_\infty((0,T),W_2^1(\Omega))$ in~\eqref{ubd1} is a straightforward consequence of~\eqref{a3PL}, parabolic regularity, and the estimate on $(v_\gamma)_{\gamma\ge 1}$ in $L_\infty((0,T),L_2(\Omega))$ in~\eqref{ubd1}, and could thus be removed from the assumption~\eqref{ubd1}. We have however chosen to include it in~\eqref{ubd1}, since it is actually (a part of) the first step of the proof of the validity of~\eqref{ubd1} performed in \Cref{prop:ue0}, see \Cref{lem:ste1} and \Cref{lem:ge2}. The boundedness of $(u_\gamma)_{\gamma\ge 1}$ and $(v_\gamma)_{\gamma\ge 1}$ in $L_\infty((0,T),L_2(\Omega))$ is actually derived afterwards, see the proof of \Cref{prop:ue0}.
\end{remark}

In the following, $c$ and $(c_i)_{i\ge 2}$ denote positive constants depending only on $\Omega$, $\theta$, $D$, $\alpha$, $(u_0,v_0,w_0)$, and $c_1$. The dependence upon additional parameters will be indicated explicitly. 

Several additional estimates can be derived from~\eqref{ubd}. 

\begin{lemma}\label{lem:bd}
	Assume~\eqref{ubd}. There is $c_2(T)>0$ such that, for all $\gamma\ge 1$,
	\begin{align}
		\int_0^T \big[  \|w_\gamma(t)\|_{W_2^2}^2 + \|u_\gamma(t)\|_4^4 + \|\nabla w_\gamma(t)\|_4^4  + \|(u_\gamma \nabla w_\gamma)(t)\|_2^2 \big]\ \mathrm{d}t & \le c_2^2(T), \label{ubd3} \\
		\int_0^T \|\partial_t n_\gamma(t)\|_{(W_2^1)'}^2\ \mathrm{d}t & \le c_2^2(T), \label{ubd4}
	\end{align}
	recalling that $n_\gamma=u_\gamma+v_\gamma$, see~\eqref{eq:ng}.
\end{lemma}

\begin{proof}
	It first readily follows from~\eqref{a3PL}, \eqref{ubd}, and parabolic regularity that
	\begin{equation}
		\int_0^T \|w_\gamma(t)\|_{W_2^2}^2\ \mathrm{d}t \le c(T). \label{ubd5}
	\end{equation}
	We next combine~\eqref{eq:GN} with~\eqref{ubd1} and~\eqref{ubd5} to find
	\begin{align}
		\int_0^T \|\nabla w_\gamma(t)\|_4^4\ \mathrm{d}t & \le c_0^4 \int_0^T \|\nabla w_\gamma(t)\|_{W_2^1}^2 \|\nabla w_\gamma(t)\|_{2}^2\ \mathrm{d}t \nonumber \\
		& \le c_0^4 c_1^2 	\int_0^T \|w_\gamma(t)\|_{W_2^2}^2\ \mathrm{d}t \le c(T). \label{ubd6}
	\end{align}
	Similarly, we infer from~\eqref{eq:GN} and~\eqref{ubd} that
	\begin{align}
		\int_0^T \|u_\gamma(t)\|_4^4\ \mathrm{d}t & \le c_0^4 \int_0^T \|u_\gamma(t)\|_{W_2^1}^2 \|u_\gamma(t)\|_{2}^2\ \mathrm{d}t \nonumber \\
		& \le c_0^4 c_1^2 \left( c_1^2 T + \int_0^T \|\nabla u_\gamma(t)\|_{2}^2\ \mathrm{d}t \right) \le c(T). \label{ubd7}
	\end{align}
	Gathering~\eqref{ubd5}, \eqref{ubd6}, and~\eqref{ubd7} and using H\"older's inequality lead us to~\eqref{ubd3}.
	
	We next consider $\varphi\in W_2^1(\Omega)$ and deduce from~\eqref{a1PL}, \eqref{a2PL}, \eqref{bPL}, and H\"older's inequality that
	\begin{equation*}
		\big| \langle \partial_t n_\gamma , \varphi \rangle_{(W_2^1)',W_2^1} \big| = \left| \int_\Omega  \nabla\varphi\cdot \big( \nabla u_\gamma - u_\gamma \nabla w_\gamma \big)\ \mathrm{d}x \right| \le \|\nabla\varphi\|_2 \big( \|\nabla u_\gamma\|_2 + \|u_\gamma \nabla w_\gamma\|_2 \big),
	\end{equation*}
	so that
	\begin{equation*}
		\|\partial_t n_\gamma \|_{(W_2^1)'} \le  \|\nabla u_\gamma\|_2 + \|u_\gamma \nabla w_\gamma\|_2
	\end{equation*}
	by a duality argument. Combining the above inequality with~\eqref{ubd2} and~\eqref{ubd3} readily gives~\eqref{ubd4} and completes the proof. 
\end{proof}

We are now ready to investigate the compactness properties of the family $(u_\gamma,v_\gamma,w_\gamma)_{\cma{\gamma\ge 1}}$. At this point, we emphasize that we are not in a position to apply directly Aubin-Lions' lemma \cite[Corollary~4]{Sim1987} to either $(n_\gamma)_{\gamma\ge 1}$ or $(v_\gamma)_{\gamma\ge 1}$, as no compactness with respect to the space variable is available, according to~\eqref{ubd1} and~\eqref{ubd4}.  Furthermore, in light of~\eqref{ubd1} and~\eqref{ubd2}, there is no information on the behaviour with respect to time of $(u_\gamma)_{\gamma\ge 1}$, which also prevents us from a direct use of \cite[Corollary~4]{Sim1987}. In order to exploit the above mentioned ``partial'' compactness properties, we make use of the following lemma \cite[Chapitre~I, Lemme~5.2]{Lio1969}, also known as Ehrling's lemma \cite[Eq.~(7)]{Ehr1954}, which is a consequence of the compactness of the embedding of $W_2^1(\Omega)$ in $L_2(\Omega)$ and the continuity of that of $L_2(\Omega)$ in $W_2^1(\Omega)'$, see also \cite[Lemma~8]{Sim1987}.

\begin{lemma}\label{lem:jll}
	Given $\eta>0$, there is $c_3(\eta)>0$ such that
	\begin{equation*}
		\|z\|_2 \le \eta \|z\|_{W_2^1} + c_3(\eta) \|z\|_{(W_2^1)'}, \qquad z\in W_2^1(\Omega).
	\end{equation*}
\end{lemma}

\begin{proposition}\label{prop:cv}
	Assume~\eqref{ubd}. There are a sequence $(\gamma_j)_{j\ge 1}$, $\gamma_j\to\infty$, and non-negative functions
	\begin{align*}
		& n\in W_2^1((0,T),W_2^1(\Omega)')\cap L_\infty((0,T),L_2(\Omega))\cap L_2((0,T),W_2^1(\Omega)), \\
		& w\in W_2^1((0,T),L_2(\Omega))\cap L_\infty((0,T),W_2^1(\Omega))\cap L_2((0,T),W_2^2(\Omega)),
	\end{align*}
	such that the convergences~\eqref{eq:cvs}, \eqref{eq:cvi1}, and~\eqref{eq:cvi2} hold true; that is,
	\begin{align*}
		& \lim_{j\to\infty} \sup_{t\in [0,T]} \big\{ \|n_{\gamma_j}(t)-n(t)\|_{(W_2^1)'} + \|w_{\gamma_j}(t) - w(t) \|_{2} \big\} = 0, \\
		& \lim_{j\to\infty} \int_0^T \big[ \|w_{\gamma_j}(t)- w(t)\|_{W_2^1}^2 + \|n_{\gamma_j}(t)-n(t)\|_2^2 + \|e_{\gamma_j}(t)\|_2^2 \big]\ \mathrm{d}t = 0,
	\end{align*}
	and
	\begin{equation*}
	\lim_{j\to\infty} \int_0^T \left[ \left\| \left( u_{\gamma_j} - \frac{\theta n}{1+ \theta} \right)(t) \right\|_2^2 + \left\| \left( v_{\gamma_j} - \frac{n}{1+ \theta} \right)(t) \right\|_2^2 \right]\ \mathrm{d}t = 0. 
	\end{equation*}
\end{proposition}

\begin{proof}
	We first observe that a straightforward consequence of~\eqref{ubd2} is that
	\begin{equation}
		\lim_{\gamma\to\infty} \int_0^T \|e_\gamma(t)\|_2^2\ \mathrm{d}t = 0. \label{eq:sce}
	\end{equation}
	
	Next, since $(n_\gamma)_{\gamma\ge 1}$ is bounded in $L_\infty((0,T),L_2(\Omega))$ by~\eqref{eq:ng} and~\eqref{ubd1} and $L_2(\Omega)$ is compactly embedded in $W_2^1(\Omega)'$, we deduce from~\eqref{ubd4} and \cite[Corollary~4]{Sim1987} that 
	\begin{equation*}
		(n_\gamma)_{\gamma\ge 1} \;\text{ is relatively compact in }\; C([0,T],W_2^1(\Omega)'). \label{eq:rcn}
	\end{equation*}
	A similar argument based on~\eqref{ubd}, \eqref{ubd3}, and the compactness of the embeddings of $W_2^1(\Omega)$ in $L_2(\Omega)$ and of $W_2^2(\Omega)$ in $W_2^1(\Omega)$ entails that
	\begin{equation*}
		(w_\gamma)_{\gamma\ge 1} \;\text{ is relatively compact in }\; C([0,T],L_2(\Omega)) \;\text{ and in }\; L_2((0,T),W_2^1(\Omega)). \label{eq:rcw}
	\end{equation*}
	Consequently, there are a sequence $(\gamma_j)_{j\ge 1}$, $\gamma_j\to\infty$, and non-negative functions
	\begin{align*}
		& n\in W_2^1((0,T),W_2^1(\Omega)') \cap L_\infty((0,T),L_2(\Omega)), \\
		& w\in W_2^1((0,T),L_2(\Omega))\cap L_\infty((0,T),W_2^1(\Omega))\cap L_2((0,T),W_2^2(\Omega))
	\end{align*}
	such that the convergences~\eqref{eq:cvs} hold true, along with
	\begin{equation}
		\lim_{j\to\infty} \int_0^T \|w_{\gamma_j}(t)- w(t)\|_{W_2^1}^2\ \mathrm{d}t = 0. \label{eq:cvw1} 
	\end{equation}
	Moreover, it readily follows from~\eqref{ubd} and~\eqref{ubd3} that there are non-negative functions
	\begin{equation}
		u\in L_\infty((0,T),L_2(\Omega))\cap L_2((0,T),W_2^1(\Omega)) \;\text{ and }\; v\in L_\infty((0,T),L_2(\Omega)) \label{eq:reguv}
	\end{equation}
	such that, up to the extraction of a further subsequence,
	\begin{subequations}\label{eq:wcv}
	\begin{align}
		(u_{\gamma_j},v_{\gamma_j})_j & \stackrel{*}{\rightharpoonup} (u,v) \;\text{ in }\; L_\infty((0,T),L_2(\Omega;\mathbb{R}^2)), \label{eq:wcv1} \\
		(u_{\gamma_j})_j & \rightharpoonup u \;\text{ in }\; L_2((0,T),W_2^1(\Omega))\cap L_4((0,T)\times\Omega), \label{eq:wcv2} \\
		(w_{\gamma_j})_j & \rightharpoonup w \;\text{ in }\; L_2((0,T),W_2^2(\Omega)). \label{eq:wcv3}
	\end{align}
\end{subequations}
Since
\begin{equation*}
	u_{\gamma_j} = \frac{\theta n_{\gamma_j} - e_{\gamma_j}}{1+\theta} \;\text{ and }\; v_{\gamma_j} = \frac{n_{\gamma_j} + e_{\gamma_j}}{1+\theta}
\end{equation*}
by~\eqref{eq:ng} and~\eqref{eq:eg}, a first consequence of~\eqref{eq:cvs} and~\eqref{eq:sce} is that
\begin{equation}
	u = \frac{\theta n}{1+\theta} \;\text{ and }\; v = \frac{n}{1+\theta} \;\;\text{ a.e. in }\;\; (0,T)\times\Omega. \label{eq:uv}
\end{equation}
In particular, it follows from~\eqref{eq:reguv}, \eqref{eq:uv}, and the positivity of $\theta$  that $n\in L_2((0,T),W_2^1(\Omega))$ and we have thus established the regularity properties of $(n,w)$ listed in \Cref{prop:cv}.

We next turn to the proof of ~\eqref{eq:cvi1} and first recall that the convergences of $(w_{\gamma_j})_{j\ge 1}$ and $(e_{\gamma_j})_{j\ge 1}$ stated therein are already shown in~\eqref{eq:sce} and~\eqref{eq:cvw1}. We are thus left with the strong convergence of $(n_{\gamma_j})_{j\ge 1}$ which will follow from \Cref{lem:jll}. Indeed, for $\eta>0$, we infer from~\eqref{eq:ng}, \eqref{eq:eg}, and \Cref{lem:jll} that
\begin{align*}
	\|n_{\gamma_j} - n \|_2 & = \|u_{\gamma_j} + v_{\gamma_j} - n\|_2 = \left\| \frac{1+\theta}{\theta} u_{\gamma_j} - n + \frac{e_{\gamma_j}}{\theta} \right\|_2 \\
	& \le \frac{1+\theta}{\theta} \left\| u_{\gamma_j} - \frac{\theta n}{1+\theta} \right\|_2 + \frac{\|e_{\gamma_j}\|_2}{\theta} \\
	& \le \frac{(1+\theta)\eta}{\theta} \left\| u_{\gamma_j} - \frac{\theta n}{1+\theta} \right\|_{W_2^1}  + \frac{(1+\theta) c_3(\eta)}{\theta} \left\| u_{\gamma_j} - \frac{\theta n}{1+\theta} \right\|_{(W_2^1)'} + \frac{\|e_{\gamma_j}\|_2}{\theta} .
\end{align*}
After integration with respect to time, we deduce from~\eqref{eq:ng}, \eqref{ubd}, and~\eqref{eq:eg} that
\begin{align*}
	\|n_{\gamma_j} - n \|_{L_2((0,T)\times\Omega)} & \le \frac{(1+\theta)\eta}{\theta} \left( \|u_{\gamma_j}\|_{L_2((0,T),W_2^1(\Omega))} + \frac{\theta}{1+\theta}  \|n\|_{L_2((0,T),W_2^1(\Omega))} \right) \\
	& \qquad + c_3(\eta) \| n_{\gamma_j} - n \|_{L_2((0,T),W_2^1(\Omega)')} + \frac{c_3(\eta)}{\theta}  \| e_{\gamma_j} \|_{L_2((0,T),W_2^1(\Omega)')} \\
	& \qquad +  \frac{\|e_{\gamma_j}\|_{L_2((0,T)\times\Omega)}}{\theta} \\
	& \le  \frac{(1+\theta)\eta}{\theta} \left( c_1 \sqrt{T} + c_1 + \frac{\theta}{1+\theta}  \|n\|_{L_2((0,T),W_2^1(\Omega))} \right) \\
	& \qquad + c_3(\eta) \| n_{\gamma_j} - n \|_{L_2((0,T),W_2^1(\Omega)')} + \frac{c_3(\eta)}{\theta}  \| e_{\gamma_j} \|_{L_2((0,T),W_2^1(\Omega)')} \\
	& \qquad +  \frac{\|e_{\gamma_j}\|_{L_2((0,T)\times\Omega)}}{\theta}.
\end{align*}
Owing to~\eqref{eq:cvs}, \eqref{eq:sce}, and the continuous embedding of $L_2(\Omega)$ in $W_2^1(\Omega)'$, we may take the limit $j\to\infty$ in the above inequality to obtain
\begin{equation*}
	\limsup_{j\to\infty} \|n_{\gamma_j} - n \|_{L_2((0,T)\times\Omega)} \le \frac{(1+\theta)\eta}{\theta} \left( c_1 \sqrt{T} + c_1 + \frac{\theta}{1+\theta}  \|n\|_{L_2((0,T),W_2^1(\Omega))} \right).
\end{equation*}
The above inequality being valid for any $\eta>0$, we conclude that
\begin{equation*}
	\lim_{j\to\infty}  \|n_{\gamma_j} - n \|_{L_2((0,T)\times\Omega)} = 0,
\end{equation*}
which completes the proof of~\eqref{eq:cvi1}. Finally, the convergences~\eqref{eq:cvi2} are immediate consequences of~\eqref{eq:ng}, \eqref{eq:eg}, \eqref{eq:cvi1}, and~\eqref{eq:sce}.
\end{proof}

\begin{corollary}\label{cor:limch}
Assume~\eqref{ubd}. There holds
\begin{equation*}
	\lim_{j\to\infty} \int_0^T \left\| \left( u_{\gamma_j}\nabla w_{\gamma_j} - \frac{\theta}{1+\theta} n \nabla w \right)(t) \right\|_{2}^{4/3}\ \mathrm{d}t = 0.
\end{equation*}
\end{corollary}

\begin{proof}
	We first recall that $\theta n/(1+\theta) = u$ by~\eqref{eq:uv}.
	
	On the one hand, thanks to~\eqref{eq:GN}, \eqref{ubd}, the regularity of $u$, and H\"older's inequality, we obtain
	\begin{align*}
		& \int_0^T \| ( u_{\gamma_j} - u )(t) \|_{4}^{2}\ \mathrm{d}t  \le c_0^2 \int_0^T \|( u_{\gamma_j} - u)(t) \|_{W_2^1} \| ( u_{\gamma_j} -u )(t) \|_{2}  \ \mathrm{d}t \\
		& \qquad \le c_0^2 \big( \|u_{\gamma_j}\|_{L_2((0,T),W_2^1(\Omega))} + \|u\|_{L_2((0,T),W_2^1(\Omega))} \big) \| u_{\gamma_j} - u \|_{L_2((0,T)\times\Omega)} \\
		& \qquad \le  c_0^2 \left( c_1 + c_1 \sqrt{T} + \|u\|_{L_2((0,T),W_2^1(\Omega))} \right) \| u_{\gamma_j} -  u \|_{L_2((0,T)\times\Omega)}.
		\end{align*}
		Hence, by~\eqref{eq:cvi2},
		\begin{equation}
			\lim_{j\to\infty}\int_0^T \| ( u_{\gamma_j} - u)(t) \|_{4}^{2}\ \mathrm{d}t = 0. \label{eq:cvu}
		\end{equation}
		On the other hand, we deduce from~\eqref{eq:GN}, \eqref{ubd3}, and H\"older's inequality that
		\begin{align*}
			& \int_0^T \|(\nabla w_{\gamma_j} - \nabla w)(t)\|_4^2\ \mathrm{d}t \\
			& \qquad \le c_0^2 \int_0^T \|(\nabla w_{\gamma_j} - \nabla w)(t)\|_{W_2^1} \|(\nabla w_{\gamma_j} - \nabla w)(t)\|_2\ \mathrm{d}t \\
			& \qquad \le c_0^2 \left( c(T) + \|w\|_{L_2((0,T),W_2^2(\Omega))} \right) \|w_{\gamma_j} - w\|_{L_2((0,T),W_2^1(\Omega))},
		\end{align*}
		and we use~\eqref{eq:cvi1} to conclude that
		\begin{equation}
			\lim_{j\to\infty} \int_0^T \|(\nabla w_{\gamma_j} - \nabla w)(t)\|_4^2\ \mathrm{d}t = 0. \label{eq:cvgw}
		\end{equation}
		Now, using H\"older's inequality, along with~\eqref{ubd3} and~\eqref{eq:wcv2}, gives
		\begin{align*}
			& \int_0^T \| ( u_{\gamma_j} \nabla w_{\gamma_j} - u \nabla w )(t) \|_{2}^{4/3}\ \mathrm{d}t \\
			& \qquad \le c \int_0^T \| [( u_{\gamma_j} - u) \nabla w_{\gamma_j} ](t) \|_{2}^{4/3}\ \mathrm{d}t + c \int_0^T \| [ u (\nabla w_{\gamma_j} - \nabla w)](t)\|_2^{4/3}\ \mathrm{d}t \\
			& \qquad \le c \int_0^T \|(u_{\gamma_j} - u)(t)\|_{4}^{4/3} \|\nabla w_{\gamma_j}(t) \|_{4}^{4/3}\ \mathrm{d}t + c \int_0^T \|u(t)\|_{4}^{4/3} \|(\nabla w_{\gamma_j} - \nabla w)(t)\|_{4}^{4/3}\ \mathrm{d}t \\
			& \qquad \le c \left( \int_0^T \|(u_{\gamma_j} - u)(t)\|_{4}^{2}\ \mathrm{d}t \right)^{2/3} \left( \int_0^T \|\nabla w_{\gamma_j}(t) \|_{4}^{4}\ \mathrm{d}t \right)^{1/3} \\
			& \qquad\qquad + c \left( \int_0^T \|u(t)\|_{4}^{4}\ \mathrm{d}t \right)^{1/3} \left( \int_0^T \|(\nabla w_{\gamma_j} - \nabla w)(t)\|_{4}^{2}\ \mathrm{d}t \right)^{2/3}  \\
			& \qquad \le c(T) \left[ \left( \int_0^T \|(u_{\gamma_j} - u)(t)\|_{4}^{2}\ \mathrm{d}t \right)^{2/3} + \left( \int_0^T \|(\nabla w_{\gamma_j} - \nabla w)(t)\|_{4}^{2}\ \mathrm{d}t \right)^{2/3} \right]
		\end{align*}
		and we may pass to the limit $j\to\infty$ in the above inequality with the help of~\eqref{eq:cvu} and~\eqref{eq:cvgw} to obtain \Cref{cor:limch}.
\end{proof}

\begin{proof}[Proof of \Cref{prop:bd}]
	Gathering the outcome of \Cref{prop:cv} and \Cref{cor:limch} gives \Cref{prop:bd}.
\end{proof}

\section{Uniform estimates}\label{sec:ue}

This section is devoted to the derivation of estimates on solutions to~\eqref{aPL}--\eqref{iPL} which do not depend on $\gamma$ and allow us to apply \Cref{prop:bd} to prove \Cref{thm:sl} in \Cref{sec:pth}. In the following, $C$ and $(C_i)_{i\ge 0}$ denote positive constants depending only on $\Omega$, $\theta$, $D$, $\alpha$, and $(u_0,v_0,w_0)$. The dependence upon additional parameters will be indicated explicitly. 
	
Two different cases are handled: first, no constraint is placed on $M=\|u_0+v_0\|_1$ and the estimates we derive are only valid on a finite time interval, a feature which complies with the possible finite time blowup of solutions to~\eqref{aKS}--\eqref{iKS}. We next assume that $M=\|u_0+v_0\|_1< 4\pi D (1+\theta)$ and show that estimates are available on any finite time interval, which is again consistent with the properties of~\eqref{aKS}--\eqref{iKS}, as global solutions to~\eqref{aKS}--\eqref{iKS} exist in that case. Specifically, we prove the following result.

\begin{proposition}\label{prop:ue0}
	Let $M>0$ and consider $(u_0,v_0,w_0)$ satisfying~\eqref{eq:ic} and~\eqref{mass0}. For $\gamma\ge 1$, we denote the corresponding solution to~\eqref{aPL}--\eqref{iPL} provided by \Cref{prop:wp} by $(u_\gamma,v_\gamma,w_\gamma)$. 
	\begin{itemize}
		\item [(a)] There is $T_\infty\in (0,\infty)$ depending only on $\Omega$, $\theta$, $D$, $\alpha$, and $(u_0,v_0,w_0)$ with the following property: for any $T\in (0,T_\infty)$, there is $C(T)>0$ depending only on $\Omega$, $\theta$, $D$, $\alpha$, $(u_0,v_0,w_0)$, and $T$ such that, for any $\gamma\ge 1$, 
		\begin{subequations}\label{eq:ue}
		\begin{align}
			\|u_\gamma(t)\|_2 + \|v_\gamma(t)\|_2 + \|w_\gamma(t)\|_{W_2^1} &\le C(T), \qquad t\in [0,T], \label{eq:ues} \\
			\int_0^T \big[ \|\nabla u_\gamma(t)\|_2^2 + \|\partial_t w_\gamma(t)\|_2^2 + \gamma \|e_\gamma(t)\|_2^2 \big] \ \mathrm{d}t & \le C(T), \label{eq:uei}
		\end{align}
		\end{subequations}
		recalling that $e_\gamma$ is defined in~\eqref{eq:eg}.
		\item [(b)] Assume further that $M\in (0,4\pi D(1+\theta))$. Then, for all $T>0$, there is $C(T)>0$ depending only on $\Omega$, $\theta$, $D$, $\alpha$, $(u_0,v_0,w_0)$, and $T$ such that the estimates~\eqref{eq:ue} are valid for all $\gamma\ge 1$.
	\end{itemize}
\end{proposition}

\subsection{Small time estimates}\label{sec:ste}

We begin with the following estimate which is valid whatever the value of $M$ but only on a finite time interval.

\begin{lemma}\label{lem:ste1}
	There is $T_\infty\in (0,\infty)$ (defined in~\eqref{eq:ste07} below) depending only on $\Omega$, $\theta$, $D$, $\alpha$, and $(u_0,v_0,w_0)$ with the following property: for any $T\in (0,T_\infty)$, there is $C_0(T)>0$ depending only on $\Omega$, $\theta$, $D$, $\alpha$, $(u_0,v_0,w_0)$, and $T$ such that, for any $\gamma \ge 1$, 
	\begin{align}
		\|L(u_\gamma(t))\|_1 + \|L_\theta(v_\gamma(t))\|_1 + \|w_\gamma(t)\|_{W_2^1} & \le C_0(T), \qquad t\in (0,T), \label{eq:ste01}\\
		\int_0^T \|\partial_t w_\gamma(t)\|_2^2\ \mathrm{d}t & \le C_0(T), \label{eq:ste02}
	\end{align}
	where
	\begin{equation}\label{eL}
		L(r) := r \ln r -r+1 \ge 0, \qquad  L_\theta (r) := \frac{L(\theta r)}{\theta} = r \ln (\theta r) -r+ \frac{1}{\theta} \ge 0, \qquad r \ge 0.
	\end{equation}
\end{lemma}

\begin{proof}
	On the one hand, it follows from~\eqref{a1PL}, \eqref{a3PL}, and~\eqref{bPL} that
	\begin{align*}
		\frac{\mathrm{d}}{\mathrm{d}t} \int_\Omega L(u_\gamma)\ \mathrm{d}x & = - \int_\Omega \nabla\ln{u_\gamma} \cdot \big( \nabla u_\gamma - u_\gamma \nabla w_\gamma \big)\ \mathrm{d}x + \gamma \int_\Omega \ln{u_\gamma} (\theta v_\gamma - u_\gamma)\ \mathrm{d}x \\
		& = - 4 \big\|\nabla\sqrt{u_\gamma}\big\|_2^2 - \int_\Omega u_\gamma \Delta w_\gamma\ \mathrm{d}x + \gamma \int_\Omega \ln{u_\gamma} (\theta v_\gamma - u_\gamma)\ \mathrm{d}x \\
		& = - 4 \big\|\nabla\sqrt{u_\gamma}\big\|_2^2 + \frac{1}{D} \int_\Omega u_\gamma \big( v_\gamma - \alpha w_\gamma - \partial_t w_\gamma \big)\ \mathrm{d}x + \gamma \int_\Omega \ln{u_\gamma} (\theta v_\gamma - u_\gamma)\ \mathrm{d}x.
	\end{align*}
	On the other hand, we infer from~\eqref{a2PL}, \eqref{a3PL}, and~\eqref{bPL} that
	\begin{align*}
		\frac{\mathrm{d}}{\mathrm{d}t} \int_\Omega L_\theta(v_\gamma)\ \mathrm{d}x & = - \gamma \int_\Omega \ln{(\theta v_\gamma)} (\theta v_\gamma - u_\gamma)\ \mathrm{d}x, \\
		\frac{1}{2\gamma} \frac{\mathrm{d}}{\mathrm{d}t} \|v_\gamma\|_2^2 & = \int_\Omega u_\gamma v_\gamma\ \mathrm{d}x - \theta \|v_\gamma\|_2^2,
	\end{align*}
	and
	\begin{align*}
		\frac{\theta}{2} \frac{\mathrm{d}}{\mathrm{d}t} \left( \|\nabla w_\gamma\|_2^2 
		+ \frac{\alpha}{D} \|w_\gamma\|_2^2 \right) & =  \frac{\theta}{D} \int_\Omega v_\gamma \partial_ t w_\gamma\ \mathrm{d}x -  \frac{\theta\|\partial_t w_\gamma\|_2^2}{D} \\
		\frac{1}{2\gamma D} \frac{\mathrm{d}}{\mathrm{d}t} \|\partial_t w_\gamma\|_2^2 & = \frac{1}{\gamma D} \int_\Omega \partial_t v_\gamma \partial_ t w_\gamma\ \mathrm{d}x - \frac{\|\nabla \partial_t w_\gamma\|_2^2}{\gamma} - \frac{ \alpha \|\partial_t w_\gamma\|_2^2}{\gamma D}.
	\end{align*}
	Introducing
	\begin{equation*}
		P_\gamma(u_\gamma,v_\gamma,w_\gamma) := \|L(u_\gamma)\|_1 + \|L_\theta(v_\gamma)\|_1 + \frac{\theta}{2} \|\nabla w_\gamma\|_2^2
		+ \frac{\alpha \theta}{2D} \|w_\gamma\|_2^2 + \frac{\|\partial_t w_\gamma\|_2^2}{2\gamma D}  + \frac{\|v_\gamma\|_2^2}{2\gamma} ,
	\end{equation*}
	we sum up the above five identities and use~\eqref{a2PL} to obtain
	\begin{align*}
		\frac{\mathrm{d}}{\mathrm{d}t}P_\gamma(u_\gamma,v_\gamma,w_\gamma) & = - 4 \big\|\nabla\sqrt{u_\gamma}\big\|_2^2 + \frac{D+1}{D} \int_\Omega u_\gamma v_\gamma\ \mathrm{d}x - \frac{\alpha}{D} \int_\Omega u_\gamma w_\gamma\ \mathrm{d}x \\
		& \qquad - \left( \frac{\theta}{D} + \frac{\alpha}{\gamma D} \right) \|\partial_t w_\gamma\|_2^2 - \frac{\|\nabla \partial_t w_\gamma\|_2^2}{\gamma} - \theta \|v_\gamma\|_2^2  \\
		& \qquad - \gamma \int_\Omega (\theta v_\gamma - u_\gamma) \ln{\left( \frac{\theta v_\gamma}{u_\gamma} \right)}\ \mathrm{d}x + \frac{1}{D} \int_\Omega \left(\frac{1}{\gamma} \partial_t v_\gamma + \theta v_\gamma - u_\gamma \right) \partial_t w_\gamma\ \mathrm{d}x
		 \\
		& = - 4 \big\|\nabla\sqrt{u_\gamma}\big\|_2^2 + \frac{D+1}{D} \int_\Omega u_\gamma v_\gamma\ \mathrm{d}x - \frac{\alpha}{D} \int_\Omega u_\gamma w_\gamma\ \mathrm{d}x 
		- \left( \frac{\theta}{D} + \frac{\alpha}{\gamma D} \right) \|\partial_t w_\gamma\|_2^2 \\
		& \qquad - \frac{\|\nabla \partial_t w_\gamma\|_2^2}{\gamma} - \theta \|v_\gamma\|_2^2  - \gamma \int_\Omega (\theta v_\gamma - u_\gamma) \ln{\left( \frac{\theta v_\gamma}{u_\gamma} \right)}\ \mathrm{d}x.
	\end{align*}
	Since both $u_\gamma$ and $w_\gamma$ are non-negative and the logarithm is an increasing function on $(0,\infty)$, we infer from Young's inequality that
	\begin{align*}
		\frac{\mathrm{d}}{\mathrm{d}t}P_\gamma(u_\gamma,v_\gamma,w_\gamma) \le  - 4 \big\|\nabla\sqrt{u_\gamma}\big\|_2^2 + \frac{(D+1)^2}{2\theta D^2} \| u_\gamma\|_2^2 + \frac{\theta}{2} \|v_\gamma\|_2^2 - \theta \|v_\gamma\|_2^2 - \frac{\theta}{D} \|\partial_t w_\gamma\|_2^2.
	\end{align*}
	Hence,
	\begin{equation}
		\frac{\mathrm{d}}{\mathrm{d}t}P_\gamma(u_\gamma,v_\gamma,w_\gamma) \le  - 4 \big\|\nabla\sqrt{u_\gamma}\big\|_2^2 + \frac{(D+1)^2}{2\theta D^2} \| u_\gamma\|_2^2 - \frac{\theta}{2} \|v_\gamma\|_2^2 - \frac{\theta}{D} \|\partial_t w_\gamma\|_2^2.
 		\label{eq:ste03}
	\end{equation}
	Now, let $K>1$ to be determined later. By the Gagliardo-Nirenberg inequality~\eqref{eq:GN},
	\begin{align*}
		\|u_\gamma\|_2^2 & = \big\| \sqrt{u_\gamma} \big\|_4^4 = \int_\Omega \big( \sqrt{u_\gamma} - K + K \big)^4 \mathbf{1}_{(K^2,\infty)}(u_\gamma)\ \mathrm{d}x + \int_\Omega u_\gamma^2  \mathbf{1}_{[0,K^2]}(u_\gamma)\ \mathrm{d}x \\
		& \le \int_\Omega \big( \big( \sqrt{u_\gamma} - K \big)_+ + K \big)^4 \mathbf{1}_{(K^2,\infty)}(u_\gamma)\ \mathrm{d}x + K^2 \int_\Omega u_\gamma  \mathbf{1}_{[0,K^2]}(u_\gamma)\ \mathrm{d}x \\
		& \le 8 \int_\Omega \big( \sqrt{u_\gamma} - K \big)_+^4 \ \mathrm{d}x + 8 K^4 \int_\Omega \mathbf{1}_{(K^2,\infty)}(u_\gamma)\ \mathrm{d}x + K^2 \int_\Omega u_\gamma  \mathbf{1}_{[0,K^2]}(u_\gamma)\ \mathrm{d}x \\
		& \le 8 c_0^4 \left( \big\| \nabla \big( \sqrt{u_\gamma} - K \big)_+ \big\|_2^2 \big\| \big( \sqrt{u_\gamma} - K\big)_+ \big\|_2^2 + \big\| \big( \sqrt{u_\gamma} - K \big)_+ \big\|_2^4 \right) \\
		& \qquad + 8 K^2 \int_\Omega u_\gamma \mathbf{1}_{(K^2,\infty)}(u_\gamma)\ \mathrm{d}x + K^2 \int_\Omega u_\gamma  \mathbf{1}_{[0,K^2]}(u_\gamma)\ \mathrm{d}x \\
		& \le 8 c_0^4 \big\| u_\gamma \mathbf{1}_{(K^2,\infty)}(u_\gamma) \big\|_1 \big\| \nabla \sqrt{u_\gamma} \big\|_2^2 + 8 c_0^4 \|u_\gamma\|_1^2 + 8 K^2 \|u_\gamma\|_1.
	\end{align*}
	Since 
	\begin{equation*}
		\big\| u_\gamma \mathbf{1}_{(K^2,\infty)}(u_\gamma) \big\|_1 \le \frac{1}{2\ln{K}} \int_\Omega u_\gamma \ln{u_\gamma} \mathbf{1}_{(K^2,\infty)}(u_\gamma)\ \mathrm{d}x \le \frac{\|L(u_\gamma)\|_1 + \|u_\gamma\|_1}{2\ln{K}},
	\end{equation*}
	we combine the above two inequalities and~\eqref{mass} to obtain
	\begin{equation*}
		\|u_\gamma\|_2^2 \le \frac{4 c_0^4 \big( \|L(u_\gamma)\|_1 + M \big)}{\ln{K}} \big\| \nabla \sqrt{u_\gamma} \big\|_2^2 + 8 M \big( c_0^4 M + K^2 \big). 
	\end{equation*}
	At this point, we choose
	\begin{equation*}
		K = \exp{\left\{ \frac{c_0^4 (D+1)^2}{\theta D^2}  \big( \|L(u_\gamma)\|_1 + M \big) \right\}}
	\end{equation*}
	in the above inequality and end up with 
	\begin{equation}
		\|u_\gamma\|_2^2 \le \frac{4\theta D^2}{(D+1) ^2} \big\|\nabla \sqrt{u_\gamma} \big\|_2^2 + 8M \left( c_0^4 M + \exp{\left\{ \frac{2 c_0^4 (D+1)^2}{\theta D^2} \big( \|L(u_\gamma)\|_1 + M \big) \right\}} \right). \label{eq:ste04}
	\end{equation}
	Since $\|L(u_\gamma)\|_1 \le P_\gamma(u_\gamma,v_\gamma,w_\gamma)$, it follows from~\eqref{eq:ste03} and~\eqref{eq:ste04} that there is $C_1>0$ such that
	\begin{equation}
		\frac{\mathrm{d}}{\mathrm{d}t}P_\gamma(u_\gamma,v_\gamma,w_\gamma) \le - 2 \big\|\nabla\sqrt{u_\gamma}\big\|_2^2 - \frac{\theta}{2} \|v_\gamma\|_2^2 - \frac{\theta}{D} \|\partial_t w_\gamma\|_2^2 + C_1 e^{C_1 P_\gamma(u_\gamma,v_\gamma,w_\gamma)}. \label{eq:ste05}
	\end{equation}
	We first infer from~\eqref{eq:ste05} that 
	\begin{equation*}
		\frac{\mathrm{d}}{\mathrm{d}t}P_\gamma(u_\gamma,v_\gamma,w_\gamma) \le C_1 e^{C_1 P_\gamma(u_\gamma,v_\gamma,w_\gamma)}, \qquad t > 0;
	\end{equation*}
	hence, after integration with respect to time,
	\begin{equation}
		P_\gamma(u_\gamma(t),v_\gamma(t),w_\gamma(t)) \le P_\gamma(u_0,v_0,w_0) - \frac{1}{C_1} \ln{\left( 1 - \frac{t}{T_\gamma} \right)}, \qquad t\in \left[ 0,T_\gamma \right), \label{eq:ste06}
	\end{equation}
	where
	\begin{equation*}
		T_\gamma := \frac{e^{- C_1 P_\gamma(u_0,v_0,w_0)}}{C_1^2} \in (0,\infty). 
	\end{equation*}
	Now, in view of $\gamma \ge 1$,
	\begin{align*}
		P_\gamma(u_0,v_0,w_0) \le P_1 (u_0,v_0,w_0) & := \|L(u_0)\|_1 + \|L_\theta(v_0)\|_1 + \frac{\theta}{2} \|\nabla w_0\|_2^2 
		+ \frac{\alpha \theta}{2D} \|w_0\|_2^2 \\
		& \hspace*{+10mm} + \frac{1}{2D} \|D \Delta w_0 - \alpha w_0 + v_0\|_2^2  + \frac{1}{2} \|v_0\|_2^2,
	\end{align*}
	so that
	\begin{equation}
		T_\gamma \ge T_\infty := \frac{e^{- C_1 P_1 (u_0,v_0,w_0)}}{C_1^2}, \label{eq:ste07} 
	\end{equation}
	and the estimates~\eqref{eq:ste01} readily follow from~\eqref{eq:ste06}, while~\eqref{eq:ste02} is a consequence of~\eqref{eq:ste05} and~\eqref{eq:ste06} after integrating the former with respect to time.
\end{proof}

\subsection{Global estimates}\label{sec:ge}

We next show that the outcome of \Cref{lem:ste1} is valid for any $T>0$ provided $M=\|u_0\|_1+\|v_0\|_1$ is appropriately small. The building block of the proof, which is actually one of the main contributions of \cite{LaSt2021}, is the construction of a Liapunov function for \eqref{aPL}--\eqref{iPL}, which we recall now. We set
\begin{equation}
\begin{split}
\mathcal{L}_\gamma(u,v,w) & := \int_\Omega \left( L(u) + L_\theta (v)- (u+v)w \right)\ \mathrm{d}x + \frac{1+\theta}{2} \left( D \|\nabla w\|_{2}^2 + \alpha \|w\|_{2}^2 \right) \\
& \qquad + \frac{1}{2\gamma} \|D\Delta w - \alpha w + v\|_{2}^2\ ,
\end{split}\label{dPL}
\end{equation}
where $L$ and $L_\theta$ are defined in~\eqref{eL}. Adapting the computation performed in \cite[Lemma~2.1]{LaSt2021} to~\eqref{aPL} leads to the following differential inequality.

\begin{proposition}\label{prop:lf}
Let $M>0$ and $\gamma>0$. Consider $(u_0,v_0,w_0)$ satisfying~\eqref{eq:ic} and~\eqref{mass0} and denote the corresponding solution to \eqref{aPL}--\eqref{iPL} given by \Cref{prop:wp} by $(u_\gamma,v_\gamma,w_\gamma)$. Then
\begin{equation}\label{e4.2.1}
\frac{\mathrm{d}}{\mathrm{d}t} \mathcal{L}_\gamma(u_\gamma,v_\gamma,w_\gamma) + \mathcal{D}_\gamma(u_\gamma,v_\gamma,w_\gamma) =0, \qquad t >0,
\end{equation}
where $\mathcal{D}_\gamma$ is non-negative and given by 
\begin{equation}
	\begin{split}
		\mathcal{D}_\gamma(u,v,w) & := \int_\Omega u |\nabla(\ln{u}-w)|^2\ \mathrm{d}x + \gamma \int_\Omega (\theta v - u)(\ln{(\theta v)} - \ln{u})\ \mathrm{d}x \\
		& \qquad + \frac{D}{\gamma} \|\nabla(D\Delta w - \alpha w + v)\|_{2}^2 + \left( 1+\theta+\frac{\alpha}{\gamma} \right) \|D\Delta w - \alpha w + v\|_{2}^2\ .
	\end{split} \label{ePL}
\end{equation}
\end{proposition}

From now on, we fix
\begin{equation}
	M \in (0,4\pi (1+\theta)D) \label{eq:ge0}
\end{equation}
and consider $(u_0,v_0,w_0)\in W_{3,+}^1(\Omega;\mathbb{R}^3)$ with $w_0\in W_2^2(\Omega)$ satisfying~\eqref{mass0}. For $\gamma>0$, we denote the corresponding solution to~\eqref{aPL}--\eqref{iPL} provided by \Cref{prop:wp} by $(u_\gamma,v_\gamma,w_\gamma)$ and set $n_\gamma := u_\gamma+v_\gamma$. By~\eqref{mass}, there holds
\begin{equation}
	\|n_\gamma(t)\|_1 = M, \qquad t\ge 0. \label{eq:ge1}
\end{equation}

We begin with a lower bound on $\mathcal{L}_\gamma(u_\gamma,v_\gamma,w_\gamma)$ which is obtained as in \cite[Lemma~4.3]{LaSt2021} with the help of the Moser-Trudinger inequality recalled in \Cref{prop:ap}.

\begin{lemma}\label{lem:ge1}
	For $\gamma>0$ and $t\ge 0$,
	\begin{align*}
		\mathcal{L}_\gamma(u_\gamma(t),v_\gamma(t),w_\gamma(t)) & \ge \frac{4\pi (1+\theta) D - M}{8\pi} \|\nabla w_\gamma(t)\|_2^2 + \frac{\alpha(1+\theta)}{2} \|w_\gamma(t)\|_2^2 \\
		&  \qquad\qquad + \frac{\|\partial_t w_\gamma(t)\|_2^2}{2\gamma} - C_2.
	\end{align*}
\end{lemma}

From \Cref{lem:ge1}, we derive the same estimates as in \Cref{lem:ste1} but valid for arbitrary positive times. 

\begin{lemma}\label{lem:ge2}
	For $\gamma\ge 1$ and $t\ge 0$, 	
	\begin{align*}
		\|L(u_\gamma(t))\|_1 + \|L_\theta(v_\gamma(t))\|_1 + \|w_\gamma(t)\|_{W_2^1} & \le C_3, \\
		\int_0^\infty \|\partial_t w_\gamma(s)\|_2^2\ \mathrm{d}s & \le C_3.
	\end{align*}
\end{lemma}

\begin{proof}
	We argue as in~\cite[Lemma~4.4]{LaSt2021}. Let $t>0$. On the one hand, it follows from \Cref{prop:lf} that
	\begin{equation*}
		\mathcal{L}_\gamma(u_\gamma(t),v_\gamma(t),w_\gamma(t)) + \int_0^t \mathcal{D}_\gamma(u_\gamma(s),v_\gamma(s),w_\gamma(s))\ \mathrm{d}s \le \mathcal{L}_\gamma(u_0,v_0,w_0) \le \mathcal{L}_1(u_0,v_0,w_0).
	\end{equation*}
	On the other hand, we infer from \Cref{lem:ge1} that
	\begin{equation*}
		\mathcal{L}_\gamma(u_\gamma(t),v_\gamma(t),w_\gamma(t)) \ge \frac{4\pi (1+\theta) D - M}{8\pi} \|\nabla w_\gamma(t)\|_2^2 + \frac{\alpha(1+\theta)}{2} \|w_\gamma(t)\|_2^2- C_2
	\end{equation*}
	and
	\begin{equation*}
		\mathcal{D}_\gamma(u_\gamma(s),v_\gamma(s),w_\gamma(s)) \ge (1+\theta) \|(D \Delta w_\gamma - \alpha w_\gamma + v_\gamma)(s)\|_2^2 = (1+\theta) \|\partial_t w_\gamma(s)\|_2^2, \qquad s\ge 0.  
	\end{equation*}
	Combining the above inequalities gives
	\begin{equation}
		\|w_\gamma(t)\|_{W_2^1}^2 + \int_0^t \|\partial_t w_\gamma(s)\|_2^2\ \mathrm{d}s \le C. \label{eq:ge4}
	\end{equation}
	Next, by~\eqref{dPL}, \Cref{prop:lf}, and Young's inequality $ab \le L(a) + e^{b}-1$,
	\begin{align*}
		& \|L(u_\gamma(t))\|_1 + \|L_\theta(v_\gamma(t))\|_1 \\
		& \qquad \le \mathcal{L}_\gamma(u_\gamma(t),v_\gamma(t),w_\gamma(t)) + \int_\Omega (u_\gamma+v_\gamma)(t) w_\gamma(t)\ \mathrm{d}x \\
		& \qquad \le \mathcal{L}_\gamma(u_0,v_0,w_0) + \int_\Omega \left[ L\left( \frac{u_\gamma(t)}{2} \right) + \frac{1}{\theta} L\left( \frac{\theta v_\gamma(t)}{2} \right) + \frac{1+\theta}{\theta} \big( e^{2w_\gamma(t)}-1 \big) \right]\ \mathrm{d}x \\
		& \qquad \le \mathcal{L}_1(u_0,v_0,w_0) + \int_\Omega \left[ \frac{L(u_\gamma(t))+1}{2} + \frac{L(\theta v_\gamma(t))+1}{2\theta}  + \frac{1+\theta}{\theta} e^{2w_\gamma(t)} \right]\ \mathrm{d}x \\
		& \qquad \le C + \frac{1}{2} \Big[ \|L(u_\gamma(t))\|_1 + \|L_\theta(v_\gamma(t))\|_1 \Big] + \frac{1+\theta}{\theta} \left\| e^{2w_\gamma(t)} \right\|_1.
	\end{align*}
	We now use the Moser-Trudinger inequality recalled in \Cref{prop:ap}, along with~\eqref{eq:ge4}, to estimate the last term on the right-hand side of the above inequality and conclude that
	\begin{equation*}
		\|L(u_\gamma(t))\|_1 + \|L_\theta(v_\gamma(t))\|_1 \le C,
	\end{equation*}
	thereby completing the proof of \Cref{lem:ge2}.
\end{proof}

\subsection{Proof of \Cref{prop:ue0}}

\begin{proof}[Proof of \Cref{prop:ue0}]
Setting $T_\infty=\infty$ when $M\in (0,4\pi D(1+\theta))$, we consider $T\in (0,T_\infty)$ and recall that \Cref{lem:ste1} and \Cref{lem:ge2} imply that, for $\gamma\ge 1$,
\begin{align}
	\|L(u_\gamma(t))\|_1 + \|L_\theta(v_\gamma(t))\|_1 + \|w_\gamma(t)\|_{W_2^1} & \le C_4(T), \qquad t\in (0,T), \label{eq:ue01}\\
	\int_0^T \|\partial_t w_\gamma(t)\|_2^2\ \mathrm{d}t & \le C_4(T). \label{eq:ue02}
\end{align} 
Let $t\in (0,T]$. On the one hand, it follows from~\eqref{aPL} that
\begin{align*}
	\frac{\mathrm{d}}{\mathrm{d}t} \Big( \|u_\gamma\|_2^2 + \theta \|v_\gamma\|_2^2 \Big) & = - 2 \|\nabla u_\gamma\|_2^2 + 2 \int_\Omega u_\gamma \nabla u_\gamma \cdot \nabla w_\gamma\ \mathrm{d}x - 2 \gamma \|e_\gamma\|_2^2 \\
	& = - 2 \|\nabla u_\gamma\|_2^2 - 2 \gamma \|e_\gamma\|_2^2 - \int_\Omega u_\gamma^2 \Delta w_\gamma\ \mathrm{d}x \\
	& = - 2 \|\nabla u_\gamma\|_2^2 - 2 \gamma \|e_\gamma\|_2^2 + \frac{1}{D} \int_\Omega u_\gamma^2 \Big( v_\gamma - \alpha w_\gamma - \partial_t w_\gamma \Big)\ \mathrm{d}x.
\end{align*}
Hence, by H\"older's inequality,
\begin{equation*}
	\frac{\mathrm{d}}{\mathrm{d}t} \Big( \|u_\gamma\|_2^2 + \theta \|v_\gamma\|_2^2 \Big) \le - 2 \|\nabla u_\gamma\|_2^2 - 2 \gamma \|e_\gamma\|_2^2 + \frac{1}{D} \int_\Omega u_\gamma^2 v_\gamma\ \mathrm{d}x + \frac{1}{D} \|u_\gamma\|_4^2 \|\partial_t w_\gamma\|_2. 
\end{equation*}
On the other hand, we deduce from~\eqref{a2PL} that
\begin{equation*}
	\frac{1}{\gamma} \frac{\mathrm{d}}{\mathrm{d}t} \|v_\gamma\|_3^3 = 3 \int_\Omega u_\gamma v_\gamma^2\ \mathrm{d}x - 3 \theta \|v_\gamma\|_3^3.
\end{equation*}
Combining the previous two inequalities and introducing 
\begin{equation*}
	Y_\gamma := 1 + \|u_\gamma\|_2^2 + \theta \|v_\gamma\|_2^2 + \frac{1}{\gamma} \|v_\gamma\|_3^3
\end{equation*}
lead us to the differential inequality
\begin{equation}
	\begin{split}
		\frac{\mathrm{d}Y_\gamma}{\mathrm{d}t} & \le - 2 \|\nabla u_\gamma\|_2^2 - 2 \gamma \|e_\gamma\|_2^2 - 3 \theta \|v_\gamma\|_3^3 \\
		& \qquad + \frac{1}{D} \int_\Omega u_\gamma^2 v_\gamma\ \mathrm{d}x + 3 \int_\Omega u_\gamma v_\gamma^2\ \mathrm{d}x + \frac{1}{D} \|u_\gamma\|_4^2 \|\partial_t w_\gamma\|_2.
	\end{split}
	\label{eq:ge7}
\end{equation}
We now infer from Young's inequality that
\begin{align*}
	\frac{1}{D} \int_\Omega u_\gamma^2 v_\gamma\ \mathrm{d}x + 3 \int_\Omega u_\gamma v_\gamma^2\ \mathrm{d}x & \le \frac{\theta}{3} \|v_\gamma\|_3^3 + \frac{2}{3} \frac{\|u_\gamma\|_3^3}{\sqrt{\theta D^3}} + \frac{2\theta}{3}  \|v_\gamma\|_3^3 + \frac{9}{\theta^2} \|u_\gamma\|_3^3 \\
	& \le \theta \|v_\gamma\|_3^3 + C \|u_\gamma\|_3^3,	
\end{align*}
while the Gagliardo-Nirenberg inequality~\eqref{eq:GN}, along with Young's inequality, entails that
\begin{align*}
	\frac{1}{D} \|u_\gamma\|_4^2 \|\partial_t w_\gamma\|_2 & \le \frac{c_0^2}{D} \|u_\gamma\|_{W_2^1} \|u_\gamma\|_2 \|\partial_t w_\gamma\|_2 \le \frac{c_0^2}{D} \Big( \|\nabla u_\gamma\|_2 \|u_\gamma\|_2 + \|u_\gamma\|_2^2 \Big) \|\partial_t w_\gamma\|_2 \\
	& \le \frac{1}{2} \|\nabla u_\gamma\|_2^2 + C \|u_\gamma\|_2^2 \|\partial_t w_\gamma\|_2^2 + C \|u_\gamma\|_2^2 \|\partial_t w_\gamma\|_2 \\
	& \le \frac{1}{2} \|\nabla u_\gamma\|_2^2 + C \|u_\gamma\|_2^2 \big( 1 + \|\partial_t w_\gamma\|_2^2 \big).
\end{align*}
Collecting the above estimates gives
\begin{equation}
	\frac{\mathrm{d}Y_\gamma}{\mathrm{d}t} + \frac{3}{2} \|\nabla u_\gamma\|_2^2 + 2 \gamma \|e_\gamma\|_2^2 + 2 \theta \|v_\gamma\|_3^3 \le C_5 \|u_\gamma\|_3^3 + C_5 \|u_\gamma\|_2^2 \big( 1 + \|\partial_t w_\gamma\|_2^2 \big). \label{eq:ge8}
\end{equation}
Finally, let $\eta>0$. By \cite[Equation~(22)]{BHN1994}, there exists $C_6(\eta)>0$ (depending actually only on $\Omega$ and $\eta$) such that
\begin{equation*}
	\|z\|_3^3 \le \eta \|z \ln{|z|}\|_1 \|z\|_{W_2^1}^2 + C_6 (\eta) \|z\|_1, \qquad z\in W_2^1(\Omega).
\end{equation*}
Since $r|\ln{r}| \le L(r) + r$ for $r\ge 0$, we deduce from~\eqref{mass}, \eqref{eq:ue01}, and the above functional inequality that
\begin{align*}
	C_5 \|u_\gamma\|_3^3 & \le \eta C_5 \|L(u_\gamma) + u_\gamma\|_1 \|u_\gamma\|_{W_2^1}^2 + C_5 C_6(\eta) \|u_\gamma\|_1 \\
	& \le \eta C_5(C_4(T)+M) \|\nabla u_\gamma\|_2^2 + \eta C_5(C_4(T)+M) \|u_\gamma\|_2^2 + M C_5 C_6(\eta).
\end{align*}
Choosing $\eta=1/[2 C_5(C_4(T)+M)]$ in the above inequality and inserting the outcome in~\eqref{eq:ge8}, we end up with
\begin{equation}
	\frac{\mathrm{d}Y_\gamma}{\mathrm{d}t} + \|\nabla u_\gamma\|_2^2 + 2 \gamma \|e_\gamma\|_2^2 + 2 \theta \|v_\gamma\|_3^3 \le 
	C_7(T) \big( 1 + \|u_\gamma\|_2^2 \big) \big( 1 + \|\partial_t w_\gamma\|_2^2 \Big). \label{eq:ge9}
\end{equation}
A first consequence of~\eqref{eq:ge9} is that
\begin{equation*}
	\frac{\mathrm{d}Y_\gamma}{\mathrm{d}t}\le C_7(T)  \big( 1 + \|\partial_t w_\gamma\|_2^2 \Big) Y_\gamma, \qquad t\in (0,T].
\end{equation*}
Hence, after integrating with respect to time and using~\eqref{eq:ue02}, we obtain
\begin{equation}
	Y_\gamma(t) \le Y_\gamma(0) \exp\left\{ C_7(T) \left( t + \int_0^t \|\partial_t w_\gamma(s)\|_2^2\ \mathrm{d}s \right)\right\} \le C e^{C_7(T) t + C_7(T) C_4(T)} \le C_8(T) \label{eq:ge10}
\end{equation}
for $t\in [0,T]$, from which~\eqref{eq:ues} follows.
	
We next integrate~\eqref{eq:ge9} with respect to time and use the non-negativity of $Y_\gamma$, along with~\eqref{eq:ue02} and~\eqref{eq:ge10}, to find
\begin{align*}
	\int_0^t \Big[ \|\nabla u_\gamma(s)\|_2^2 + \gamma \|e_\gamma(s)\|_2^2 \Big]\ \mathrm{d}s & \le Y_\gamma(0) + C_7(T) \sup_{s\in [0,t]}\{1+ \|u_\gamma(s)\|_2^2\} \left( t + \int_0^t \|\partial_t w_\gamma(s)\|_2^2\ \mathrm{d}s \right) \\
	& \le C + C_7(T) C_8(T) (t + C_4(T)),
\end{align*}
and thus complete the proof.
\end{proof}

\section{Convergence: proof of \Cref{thm:sl}}\label{sec:pth}

\begin{proof}[Proof of \Cref{thm:sl}]
	Let $M>0$ and consider initial conditions $(u_0,v_0,w_0)$ satisfying~\eqref{eq:ic} and~\eqref{mass0}. For $\gamma>0$, we denote the solution to \eqref{aPL}--\eqref{iPL} given by \Cref{prop:wp} by $(u_\gamma,v_\gamma,w_\gamma)$ and set $n_\gamma = u_\gamma + v_\gamma$ and $e_\gamma = \theta v_\gamma - u_\gamma$.	According to \Cref{prop:ue0}, there is $T_\infty\in (0,\infty]$ such that the estimates~\eqref{eq:ue} hold true for all $T\in (0,T_\infty)$, having set $T_\infty=\infty$ in Case~(b). We are thus in a position to apply \Cref{prop:bd} for any $T\in (0,T_\infty)$ and use a diagonal process to construct a sequence $(\gamma_j)_{j\ge 1}$, $\gamma_j\to\infty$, and non-negative functions
	\begin{align*}
		& n\in W_{2,\mathrm{loc}}^1([0,T_\infty),W_2^1(\Omega)') \cap L_{\infty,\mathrm{loc}}([0,T_\infty),L_2(\Omega))\cap L_{2,\mathrm{loc}}([0,T_\infty),W_2^1(\Omega)), \\
		& w\in W_{2,\mathrm{loc}}^1([0,T_\infty),L_2(\Omega))\cap L_{\infty,\mathrm{loc}}([0,T_\infty),W_2^1(\Omega)) \cap L_{2,\mathrm{loc}}([0,T_\infty),W_2^2(\Omega)),
	\end{align*}
	such that the convergences~\eqref{eq:cvs}, \eqref{eq:cvi1}, \eqref{eq:cvi2}, and~\eqref{eq:cvi3} hold true for any $T\in (0,T_\infty)$. 
	
	Now, it readily follows from~\eqref{a1PL}, \eqref{a2PL}, and \eqref{bPL} that, for $j\ge 1$, $t\in (0,T_\infty)$, and $\varphi\in W_2^1(\Omega)$, 
	\begin{equation*}
		\int_0^t \langle \partial_t n_{\gamma_j}(s) , \varphi \rangle_{(W_2^1)',W_2^1} \mathrm{d}s + \int_0^t \int_\Omega \nabla\varphi \cdot \big( \nabla u_{\gamma_j} - u_{\gamma_j}
		\nabla w_{\gamma_j} \big)(s)\ \mathrm{d}x\mathrm{d}s = 0.
	\end{equation*}
	Owing to~\eqref{eq:cvs}, \eqref{eq:cvi2}, and~\eqref{eq:cvi3}, it is straightforward to pass to the limit $j\to\infty$ in the above identity and conclude that $n$ solves
	\begin{equation}
		\int_0^t \langle \partial_t n(s) , \varphi \rangle_{(W_2^1)',W_2^1} \mathrm{d}s +  \frac{\theta}{1+\theta} \int_0^t \int_\Omega \nabla\varphi \cdot \big( \nabla n - n \nabla w \big)(s)\ \mathrm{d}x\mathrm{d}s = 0 \label{eq:weakn}
	\end{equation}
	for all $t\in (0,T_\infty)$ and $\varphi\in W_2^1(\Omega)$. Similarly, we infer from~\eqref{a3PL}, \eqref{eq:cvs}, \eqref{eq:cvi1}, and~\eqref{eq:cvi2} that
	\begin{equation}
		\frac{\mathrm{d}}{\mathrm{d}t} \int_\Omega w \varphi\ \mathrm{d}x + D \int_\Omega \nabla w\cdot \nabla\varphi\ \mathrm{d}x + \alpha \int_\Omega w \varphi\ \mathrm{d}x = \frac{1}{1+\theta} \int_\Omega n \varphi\ \mathrm{d}x \label{eq:weakw}
	\end{equation}
	for all $t\in (0,T_\infty)$ and $\varphi\in W_2^1(\Omega)$. Actually, the regularity of $w$ and $n$, together with~\eqref{eq:weakw}, implies that $w$ is a strong solution to 
	\begin{equation}
	\begin{array}{ll}
		\partial_t w - D \Delta w + \alpha w  = \displaystyle{\frac{n}{1+\theta}} & \qquad\text{ in }\;\; (0,T_\infty) \times \Omega,\\[2mm]
		 \nabla w \cdot \mathbf{n} = 0 & \qquad\text{ on }\;\; (0,T_\infty) \times \partial \Omega, \\[2mm]
		 w(0) = w_0 & \qquad\text{ in }\;\; \Omega.
	\end{array}\label{eq:strongw}
	\end{equation}
	In other words, $(n,w)$ is a weak-strong solution to~\eqref{aKS}--\eqref{iKS} on $(0,T)$ in the sense of \Cref{def:c1} below for all $T\in (0,T_\infty)$, which is unique according to \Cref{prop:c2} below. A classical argument then implies that the convergences~\eqref{eq:cvs}, \eqref{eq:cvi1}, \eqref{eq:cvi2}, and~\eqref{eq:cvi3} hold true for all $\gamma\ge 1$ and $T\in (0,T_\infty)$.
	
We are left with improving the convergence of $(w_\gamma)_{\gamma\ge 1}$ from the $L_2$-convergence derived in \Cref{prop:cv} to the $W_2^1$-convergence~\eqref{cvw} claimed in \Cref{thm:sl}. To this end, we infer from~\eqref{a3PL}, \eqref{bPL}, \eqref{eq:strongw}, and H\"older's and Young's inequalities that
	\begin{align*}
		\big\| \partial_t (w_\gamma-w)\big\|_2^2 & + \frac{1}{2} \frac{\mathrm{d}}{\mathrm{d}t} \big[ D \|\nabla(w_\gamma-w)\|_2^2 + \alpha \|w_\gamma-w\|_2^2 \big] \\
		& \hspace{2cm} = \int_\Omega \left( v_\gamma - \frac{n}{1+\theta} \right) \partial_t (w_\gamma-w)\ \mathrm{d}x \\
		& \hspace{2cm} \le \frac{1}{2} \big\| \partial_t (w_\gamma-w)\big\|_2^2 + \frac{1}{2} \left\| v_\gamma - \frac{n}{1+\theta} \right\|_2^2.
		\end{align*}
	Hence,
	\begin{equation*}
		\frac{\mathrm{d}}{\mathrm{d}t} \big[ D \|\nabla(w_\gamma-w)\|_2^2 + \alpha \|w_\gamma-w\|_2^2 \big] \le  \left\| v_\gamma - \frac{n}{1+\theta} \right\|_2^2,
	\end{equation*}
	from which we deduce, after integrating with respect to time, that
	\begin{equation*}
		D \|\nabla(w_\gamma-w)(t)\|_2^2 + \alpha \|(w_\gamma-w)(t)\|_2^2 \le \int_0^t  \left\| \left( v_\gamma - \frac{n}{1+\theta} \right)(s) \right\|_2^2\ \mathrm{d}s 
	\end{equation*}
	for $t\in [0,T_\infty)$. Consequently, for $T\in (0,T_\infty)$,
	\begin{equation*}
		\sup_{t\in [0,T]} \big[ D \|\nabla(w_\gamma-w)(t)\|_2^2 + \alpha \|(w_\gamma-w)(t)\|_2^2 \big] \le \int_0^T  \left\| \left( v_\gamma - \frac{n}{1+\theta} \right)(s) \right\|_2^2\ \mathrm{d}s.
	\end{equation*}
	Since the right-hand side of the above inequality converges to zero as $\gamma\to\infty$ by~\eqref{eq:cvi2}, we readily conclude that
	\begin{equation*}
		\lim_{\gamma\to\infty} \sup_{t\in [0,T]} \big[ D \|\nabla(w_\gamma-w)(t)\|_2^2 + \alpha \|(w_\gamma-w)(t)\|_2^2 \big] = 0.
	\end{equation*}
	The convergence~\eqref{cvw} then follows due to the positivity of~$D$ and~$\alpha$, and the proof of  \Cref{thm:sl} is complete.
\end{proof}

\section{Unboundedness: proof of \Cref{thm:ub}}\label{sec:ub}

The main step of the proof of \Cref{thm:ub} is the following result asserting that  the estimates~\eqref{ubd} can be derived from a bound on the $L_2$-norm of $v_\gamma$, whenever available.

\begin{lemma}\label{lem:ub1}
	Consider $(u_0,v_0,w_0) \in W_{3,+}^1(\Omega;\mathbb{R}^3)$ and, for $\gamma>0$, denote the solution to \eqref{aPL}--\eqref{iPL} given by \Cref{prop:wp} by $(u_\gamma,v_\gamma,w_\gamma)$. We assume that there are $T\in (0,\infty)$ and $K>0$ such that
	\begin{equation}
		\|v_\gamma(t)\|_2 \le K \qquad \mbox{for all } t\in (0,T), \ \gamma\ge 1. \label{eq:e20}
	\end{equation}
	Then $(u_\gamma,v_\gamma,w_\gamma)$ satisfies~\eqref{ubd} with $e_\gamma = \theta v_\gamma - u_\gamma$.
\end{lemma}

\begin{proof}
	In the following, $C$ denote positive constants depending only on $\Omega$, $\theta$, $D$, $\alpha$, $u_0$, $v_0$, $w_0$, $T$, and $K$.
	
	It readily follows from~\eqref{a3PL}, \eqref{bPL}, \eqref{eq:e20}, and Young's inequality that
	\begin{align*}
		\|\partial_t w_\gamma\|_2^2 & + \frac{1}{2} \frac{\mathrm{d}}{\mathrm{d}t} \big[ D \|\nabla w_\gamma\|_2^2 + \alpha \|w_\gamma\|_2^2 \big] = \int_\Omega v_\gamma \partial_t w_\gamma\ \mathrm{d}x \\ 
		& \le \frac{1}{2} \|v_\gamma\|_2^2 + \frac{1}{2} \|\partial_t w_\gamma\|_2^2 \le \frac{K^2}{2} + \frac{1}{2} \|\partial_t w_\gamma\|_2^2.
	\end{align*}
	Hence,
	\begin{equation*}
			\|\partial_t w_\gamma\|_2^2 + \frac{\mathrm{d}}{\mathrm{d}t} \big[ D \|\nabla w_\gamma\|_2^2 + \alpha \|w_\gamma\|_2^2 \big] \le K^2,
	\end{equation*}
	from which we deduce, after integration with respect to time,
	\begin{equation}
		\sup_{t\in [0,T]} \|w_\gamma(t)\|_{W_2^1}^2 + \int_0^T \|\partial_t w_\gamma(t)\|_2^2\ \mathrm{d}t \le C. \label{eq:c21}
	\end{equation}
	We next infer from~\eqref{a1PL}, \eqref{a2PL}, \eqref{bPL}, the Gagliardo-Nirenberg inequality~\eqref{eq:GN}, and H\"older's inequality that
	\begin{align*}
		\frac{1}{2} \frac{\mathrm{d}}{\mathrm{d}t} \big[ \|u_\gamma\|_2^2 + \theta \|v_\gamma\|_2^2 \big] & = - \int_\Omega \nabla u_\gamma \cdot \big( \nabla u_\gamma - u_\gamma \nabla w_\gamma)\ \mathrm{d}x - \gamma \|e_\gamma\|_2^2 \\
		& = -\|\nabla u_\gamma\|_2^2 - \frac{1}{2} \int_\Omega u_\gamma^2 \Delta w_\gamma\ \mathrm{d}x - \gamma \|e_\gamma\|_2^2 \\
		& \le - \|u_\gamma\|_{W_2^1}^2 + \|u_\gamma\|_2^2 + \frac{1}{2} \| u_\gamma\|_4^2 \|\Delta w_\gamma\|_2 - \gamma \|e_\gamma\|_2^2 \\
		& \le - \|u_\gamma\|_{W_2^1}^2 + \|u_\gamma\|_2^2 + \frac{c_0^2}{2} \|u_\gamma\|_{W_2^1} \| u_\gamma\|_2 \|\Delta w_\gamma\|_2 - \gamma \|e_\gamma\|_2^2 .
	\end{align*}
	Using now~\eqref{a3PL}, Young's inequality, and \eqref{eq:e20}, we further obtain
	\begin{align}
		\frac{1}{2} \frac{\mathrm{d}}{\mathrm{d}t} \big[ \|u_\gamma\|_2^2 + \theta \|v_\gamma\|_2^2 \big] & \le - \|u_\gamma\|_{W_2^1}^2 + \|u_\gamma\|_2^2 + \frac{1}{2} \|u_\gamma\|_{W_2^1}^2 \nonumber\\
		& \qquad + \frac{c_0^4}{8 D^2} \|u_\gamma\|_2^2 \|\partial_t w_\gamma + \alpha w_\gamma - v_\gamma\|_2^2 - \gamma \|e_\gamma\|_2^2 \nonumber \\
		& \le - \frac{1}{2} \|u_\gamma\|_{W_2^1}^2 + C \|u_\gamma\|_2^2\big( 1 + \|\partial_t w_\gamma\|_2^2 + \| w_\gamma \|_2^2 + \| v_\gamma\|_2^2 \big) - \gamma \|e_\gamma\|_2^2 \nonumber \\
		& \le - \frac{1}{2} \|u_\gamma\|_{W_2^1}^2 - \gamma \|e_\gamma\|_2^2 + C \|u_\gamma\|_2^2\big( 1 + \|\partial_t w_\gamma\|_2^2 + \| w_\gamma \|_2^2 \big). \label{eq:c21a}
	\end{align}
	As this yields
	\begin{equation*}
	  \frac{\mathrm{d}}{\mathrm{d}t} \big[ \|u_\gamma\|_2^2 + \theta \|v_\gamma\|_2^2 \big] 
		\le C \big(\|u_\gamma\|_2^2+ \theta \|v_\gamma\|_2^2 \big) \big( 1 + \|\partial_t w_\gamma\|_2^2 + \| w_\gamma \|_2^2 \big),
	\end{equation*}
	Gronwall's lemma in conjunction with \eqref{eq:c21} implies that, for all $t\in [0,T]$,
	\begin{equation}
	  \|u_\gamma (t)\|_2^2 + \theta \|v_\gamma (t)\|_2^2 \le C \exp \left\{ C \int_0^T \big(1 + \|\partial_t w_\gamma (s)\|_2^2 + \| w_\gamma (s)\|_2^2 \big) \ \mathrm{d}s \right\} \le C. \label{eq:c21b}
	\end{equation}
	Integrating \eqref{eq:c21a} with respect to time, we deduce from~\eqref{eq:c21} and~\eqref{eq:c21b} that 
	\begin{equation}
		\int_0^T \big[ \|u_\gamma(t)\|_{W_2^1}^2 + \gamma \|e_\gamma(t)\|_2^2 \big]\ \mathrm{d}t \le C. \label{eq:c22}
	\end{equation}
	Collecting~\eqref{eq:e20}, \eqref{eq:c21}, \eqref{eq:c21b}, and~\eqref{eq:c22} shows that $(u_\gamma,v_\gamma,w_\gamma)$ satisfies~\eqref{ubd} on $(0,T)$ for all $\gamma\ge 1$, as claimed.
\end{proof}

We further show that a corresponding bound on $u_\gamma$ implies the bound \eqref{eq:e20} on $v_\gamma$. Here we adapt the idea from \cite[Lemma~4.11]{PaWi2023}.
\begin{lemma}\label{lem:ub2}
	Consider $(u_0,v_0,w_0) \in W_{3,+}^1(\Omega;\mathbb{R}^3)$ and, for $\gamma>0$, denote the solution to \eqref{aPL}--\eqref{iPL} given by \Cref{prop:wp} by $(u_\gamma,v_\gamma,w_\gamma)$. Assume that there are $T\in (0,\infty)$ and $\widetilde{K}>0$ such that
	\begin{equation}
	  \|u_\gamma(t)\|_2 \le \widetilde{K} \qquad \mbox{for all } t\in (0,T), \ \gamma\ge 1. \label{eq:e23}
	\end{equation} 
	Then there is $K >0$ such that \eqref{eq:e20} is fulfilled.
\end{lemma}

\begin{proof}
  Using \eqref{a2PL}, H\"{o}lder's inequality, and \eqref{eq:e23}, we have for $t \in (0,T)$
	\begin{equation*}
	  \frac{\mathrm{d}}{\mathrm{d}t} \|v_\gamma\|_2^2  = 2\gamma \left(\int_\Omega u_\gamma v_\gamma\ \mathrm{d}x 
		- \theta \|v_\gamma\|_2^2 \right) 
		\le 2\gamma \left( \|u_\gamma\|_2 \|v_\gamma\|_2 
		- \theta \|v_\gamma\|_2^2 \right) 
		\le 2\gamma \left(\widetilde{K} - \theta \|v_\gamma\|_2 \right) \|v_\gamma\|_2.
	\end{equation*}
	In view of $v_\gamma (0) = v_0$, this implies \eqref{eq:e20} with $K := \max \left\{ \frac{\widetilde{K}}{\theta}, \|v_0\|_2 \right\}$ by the comparison principle.
\end{proof}

\begin{proof}[Proof of \Cref{thm:ub}]
	It first follows from~\eqref{eq:i1} and \cite[Theorem~3.6]{HoWa2001} that
	\begin{equation}
		\lim_{t\to T_{bu}} \|n(t)\|_2=\infty. \label{eq:e1}
	\end{equation}
	Next, let $T>T_{bu}$ and assume for contradiction that there is $K>0$ such that
	\begin{equation}
		\|v_\gamma(t)\|_2 \le K \qquad \mbox{for all } t\in (0,T), \ \gamma\ge 1. \label{eq:e2}
	\end{equation}
	Then \eqref{ubd} is fulfilled by \Cref{lem:ub1}. Hence, by \Cref{prop:bd} and the proof of \Cref{thm:sl}, $(n_\gamma, 
	w_\gamma)$ converges as $\gamma \to \infty$ to a weak-strong solution $(\tilde{n},\tilde{w})$ to \eqref{aKS}--\eqref{iKS} on $(0,T)$. 
	As the classical solution $(n,w)$ to \eqref{aKS}--\eqref{iKS} is a weak-strong solution and the latter is unique by \Cref{prop:c2},
	we have $(n,w) = (\tilde{n}, \tilde{w})$. Consequently, $n \in L_\infty ((0,T), L_2(\Omega))$ which contradicts \eqref{eq:e1}
	in view of $T > T_{bu}$. Hence, \eqref{eq:e2} cannot be valid. In view of \Cref{lem:ub2}, this further implies that
	\begin{equation*}
	   \|u_\gamma(t)\|_2 \le \widetilde{K} \qquad \mbox{for all } t\in (0,T), \ \gamma\ge 1
	\end{equation*}
	cannot be satisfied for some $\widetilde{K}>0$, and the proof of \Cref{thm:ub} is complete. 	
\end{proof}

\appendix
\section{The Moser-Trudinger inequality}\label{sec:mti}

We recall here the version of the Moser-Trudinger inequality derived in \cite[Corollary~2.7]{GaZa1998} and in \cite[Section~2]{NSY1997} from \cite[Proposition~2.3]{ChYa1988} and which is used in the proof of \Cref{lem:ge2}.

\begin{proposition}\label{prop:ap}
	There is $K_0 >0$ depending only on $\Omega$ such that, for all $z\in W_2^1(\Omega)$, 
	\begin{equation*}
		\int_\Omega e^{|z|} \ \mathrm{d}x \le K_0 \, \exp \left(\frac{\| \nabla z\|_{2}^2}{8\pi} + \frac{\| z\|_{1}}{|\Omega|} \right).
\end{equation*}	
\end{proposition}

\section{Rescaling~\eqref{aKS}}\label{sec:resc}

Let $(n,w)$ be a solution to~\eqref{aKS} with initial condition $(n_0,w_0)$. Performing a simple scaling transforms~\eqref{aKS} to a version of the parabolic-parabolic chemotaxis Keller-Segel system which is usually used in the literature. Indeed, introducing 
\begin{equation}
	U(s,x) := \frac{1}{\theta} n\left( \frac{1+\theta}{\theta} s,x \right), \quad V(s,x) := w\left( \frac{1+\theta}{\theta} s,x \right), \qquad (s,x)\in (0,\infty)\times\Omega, \label{rescal}
\end{equation}
we readily deduce from~\eqref{aKS} that $(U,V)$ solves the Keller-Segel system
\begin{subequations}\label{rKS}
	\begin{align}
		\partial_s U & = \mathrm{div}(\nabla U - U \nabla V) &\text{ in }\;\; (0,\infty) \times \Omega, \label{r1KS}\\
		\frac{\theta}{(1+\theta)D} \partial_s V & = \Delta w - \frac{\alpha}{D} V + \frac{\theta}{(1+\theta)D} U & \text{ in }\;\; (0,\infty) \times \Omega, \label{r2KS} \\
		& \big( \nabla U - U \nabla V \big) \cdot \mathbf{n} = \nabla V \cdot \mathbf{n} =0 & \text{ on }\;\; (0,\infty) \times \partial \Omega, \label{r3KS} \\
		(U,V)(0) & = \left( \frac{n_0}{\theta},w_0 \right) &\text{ in }\;\; \Omega. \label{r4KS}
	\end{align}
\end{subequations}
According to \cite{GaZa1998, NSY1997}, global existence holds true for~\eqref{rKS} if 
\begin{equation*}
	\|n_0\|_1 = \theta \frac{\|n_0\|_1}{\theta} < \theta \frac{4\pi (1+\theta) D}{\theta} = 4\pi (1+\theta) D,
\end{equation*}
while finite time blowup may occur when $\|n_0\|_1$ exceeds the above threshold \cite{Ho2001, Ho2002, HoWa2001}.

\section{Uniqueness of weak-strong solutions to~\eqref{aKS}--\eqref{iKS}}\label{sec:uws}

We provide here a proof of the uniqueness result of weak-strong solutions to~\eqref{aKS}--\eqref{iKS} which is used in the proof of \Cref{thm:sl} to obtain the convergence of the whole family $(n_\gamma)_{\gamma\ge 1}$. We first make precise the meaning of weak-strong solutions to~\eqref{aKS}--\eqref{iKS}.

\begin{definition}\label{def:c1}
	Let $T\in (0,\infty)$ and $(n_0,w_0)\in L_{2,+}(\Omega)\times W_{2,+}^1(\Omega)$. A weak-strong solution to~\eqref{aKS}--\eqref{iKS} on $[0,T]$ is a pair of functions
	\begin{subequations}\label{eq:ws}
	\begin{align}
		& n\in W_{2}^1((0,T),W_2^1(\Omega)') \cap L_{\infty}((0,T),L_{2,+}(\Omega))\cap L_{2}((0,T),W_2^1(\Omega)), \quad n(0)=n_0, \label{eq:ws1} \\
		& w\in W_{2}^1((0,T),L_2(\Omega))\cap L_{\infty}((0,T),W_{2,+}^1(\Omega)) \cap L_{2}((0,T),W_2^2(\Omega)), \label{eq:ws2}
	\end{align}
	such that
	\begin{equation}
		\int_0^t \langle \partial_t n(s) , \varphi \rangle_{(W_2^1)',W_2^1} \mathrm{d}s + \frac{\theta}{1+\theta} \int_0^t 	\int_\Omega \nabla\varphi \cdot \big( \nabla n - n \nabla w \big)(s)\ \mathrm{d}x\mathrm{d}s = 0 \label{eq:ws3}
	\end{equation}
	for all $t\in (0,T)$ and $\varphi\in W_2^1(\Omega)$, while $w$ is a strong solution to 
	\begin{equation}
	\begin{array}{ll}
		\partial_t w - D \Delta w + \alpha w = \displaystyle{\frac{n}{1+\theta}} & \qquad\text{ in }\;\;  (0,T) \times \Omega,\\[2mm]
		\nabla w \cdot \mathbf{n} = 0 & \qquad\text{ on }\;\; (0,T) \times \partial \Omega, \\[2mm]
		w(0) = w_0 & \qquad\text{ in }\;\; \Omega.
	\end{array}\label{eq:ws4}
	\end{equation}
	\end{subequations}
\end{definition}

Clearly, classical solutions to~\eqref{aKS} are weak-strong solutions to~\eqref{aKS}, see \cite[Theorem~3.1]{HoWi2005} for instance for the existence of the former. Also, the weak solutions to~\eqref{aKS} constructed in \cite[Theorem~1]{Bil1998} for $(n_0,w_0)\in L_{q,+}(\Omega)\times W_{q,+}^1(\Omega)$, $q>2$, in \cite[Theorem~3.3]{GaZa1998} for $(n_0,w_0)\in L_{\infty,+}(\Omega)\times W_{q,+}^1(\Omega)$, $q>2$, in \cite{Yag1997} for $(n_0,w_0)\in W_{2,+}^s(\Omega;\mathbb{R}^2)$, $s>1$, and in \cite[Chapter~12]{Yag2010} for $(n_0,w_0)\in L_{2,+}(\Omega)\times W_{2,\mathcal{B},+}^2(\Omega)$ are weak-strong solutions to~\eqref{aKS}. Uniqueness results are also provided in \cite{GaZa1998, HoWi2005, Yag1997, Yag2010} but require more regularity than that stated in \Cref{def:c1}. Let us also mention that well-posedness in $L_{p,+}(\Omega)\times W_{q,+}^1(\Omega)$ with $p>1$ and $q>2$ is obtained in \cite{Bil1998}. We shall show now that the assumptions in \Cref{def:c1} are sufficient to guarantee uniqueness. To this end, we adapt the uniqueness argument used in the proof of \cite[Theorem~3.1]{HoWi2005} and rely heavily on the two-dimensional setting and the Gagliardo-Nirenberg inequality~\eqref{eq:GN}.

\begin{proposition}\label{prop:c2}
	Let $T\in (0,\infty)$ and $(n_0,w_0)\in L_{2,+}(\Omega)\times W_{2,+}^1(\Omega)$, where $\Omega \subset \mathbb{R}^2$ is a bounded domain with smooth boundary. There is at most one weak-strong solution to~\eqref{aKS}--\eqref{iKS} on $[0,T]$.
\end{proposition}

\begin{proof}
	Let $T>0$, $(n_0,w_0)\in L_{2,+}(\Omega)\times W_{2,+}^1(\Omega)$, and consider two weak-strong solutions $(n_1,w_1)$ and $(n_2,w_2)$ to~\eqref{aKS}--\eqref{iKS} on $(0,T)$. We set $N:= n_1-n_2$, $W:=w_1-w_2$, $\psi := \|n_2\|_{W_2^1}^2 + \|w_1\|_{W_2^2}^2$ and infer from~\eqref{eq:ws1} and~\eqref{eq:ws2} that
	\begin{equation}
		\kappa := \sup_{t\in (0,T)}\left[ \|n_2(t)\|_2 + \|w_1(t)\|_{W_2^1} \right] < \infty \;\;\text{ and } \psi\in L_1((0,T)). \label{eq:c1}
	\end{equation}
	In the following, $C$ denote positive constants depending only on $\Omega$, $\theta$, $D$, $\alpha$, $n_0$, $w_0$, $T$, and $\kappa$.
	
	On the one hand, it follows from~\eqref{eq:ws1}, \eqref{eq:ws3}, and H\"older's inequality that
	\begin{align*}
		\frac{1}{2} \frac{\mathrm{d}}{\mathrm{d}t} \|N\|_2^2 & = \big\langle \partial_t N , N \big\rangle_{(W_2^1)',W_2^1} = - \frac{\theta}{1+\theta} \int_\Omega \nabla N \cdot \big( \nabla N - n_1 \nabla w_1 + n_2 \nabla w_2 \big)\ \mathrm{d}x	\\
		& = - \frac{\theta}{1+\theta} \|\nabla N\|_2^2 + \frac{\theta}{1+\theta} \int_\Omega N \nabla N\cdot \nabla w_1\ \mathrm{d}x + \frac{\theta}{1+\theta} \int_\Omega n_2 \nabla N \cdot \nabla W\ \mathrm{d}x \\
		& \le - \frac{\theta}{1+\theta} \|\nabla N\|_2^2 + \frac{\theta}{1+\theta} \| N\|_4 \|\nabla N\|_2 \|\nabla w_1\|_4 + \frac{\theta}{1+\theta} \| n_2\|_4 \|\nabla N\|_2 \|\nabla W\|_4.
	\end{align*}
	We now use the Gagliardo-Nirenberg inequality~\eqref{eq:GN} to estimate the $L_4$-norms, along with Young's inequality, to obtain
	\begin{align*}
		\frac{1+\theta}{2\theta} \frac{\mathrm{d}}{\mathrm{d}t} \|N\|_2^2 + \|N\|_{W_2^1}^2 & \le \|N\|_{2}^2 + c_0^2 \|N\|_{W_2^1}^{1/2} \|N\|_2^{1/2} \|\nabla N\|_2 \|\nabla w_1\|_{W_2^1}^{1/2} \|\nabla w_1\|_2^{1/2}  \\
		& \qquad + c_0^2 \|n_2\|_{W_2^1}^{1/2} \|n_2\|_2^{1/2} \|\nabla N\|_2 \|\nabla W\|_{W_2^1}^{1/2} \|\nabla W\|_2^{1/2} \\
		& \le \|N\|_2^2 + c_0^2 \kappa^{1/2} \|N\|_2^{1/2} \|N\|_{W_2^1}^{3/2} \|w_1\|_{W_2^2}^{1/2} \\
		& \qquad +  c_0^2 \kappa^{1/2} \|n_2\|_{W_2^1}^{1/2}  \|N\|_{W_2^1} \|W\|_{W_2^2}^{1/2} \|W\|_{W_2^1}^{1/2} \\
		& \le C \big( 1 + \|w_1\|_{W_2^2}^2 \big) \|N\|_2^2 + \frac{1}{4}  \|N\|_{W_2^1}^2 + C \|n_2\|_{W_2^1} \|W\|_{W_2^2} \|W\|_{W_2^1}.
	\end{align*}
	Hence,
	\begin{equation}
		\frac{1+\theta}{\theta} \frac{\mathrm{d}}{\mathrm{d}t} \|N\|_2^2 + \frac{3}{2} \|N\|_{W_2^1}^2 \le C \big( 1 + \psi \big) \|N\|_2^2 + C \|n_2\|_{W_2^1} \|W\|_{W_2^2} \|W\|_{W_2^1}. \label{eq:c2}
	\end{equation}
	On the other hand, by~\eqref{eq:ws2}, \eqref{eq:ws4}, and H\"older's and Young's inequalities,
	\begin{align*}
		\frac{1}{2} \frac{\mathrm{d}}{\mathrm{d}t} \|W\|_{W_2^1}^2 & + D \big( \|\nabla W\|_2^2 + \|\Delta W\|_2^2 \big) + \alpha \|W\|_{W_2^1}^2 = \frac{1}{1+\theta} \int_\Omega \big( N W + \nabla N \cdot \nabla W \big)\ \mathrm{d}x \\
		& \le \frac{1}{2(1+\theta)} \big( \|N\|_2^2 + \|W\|_2^2) \big) + \frac{1}{4} \|\nabla N\|_2^2 + \frac{1}{(1+\theta)^2} \|\nabla W\|_2^2.
	\end{align*}
	Since there is $\delta>0$ depending only on $\Omega$, $D$, and $\alpha$ such that
	\begin{equation*}
		\delta \|z\|_{W_2^2}^2 \le D \big( \|\nabla z\|_2^2 + \|\Delta z\|_2^2 \big) + \alpha \|z\|_{W_2^1}^2, \qquad z\in W_2^2(\Omega),
	\end{equation*}
	we further obtain
	\begin{equation}
		 \frac{\mathrm{d}}{\mathrm{d}t} \|W\|_{W_2^1}^2 + 2 \delta \|W\|_{W_2^2}^2 \le \frac{1}{2} \| N\|_{W_2^1}^2 + C \big( \|N\|_2^2 + \|W\|_{W_2^1}^2 \big). \label{eq:c3}
	\end{equation}
	Combining~\eqref{eq:c2} and~\eqref{eq:c3} and using once more Young's inequality give
	\begin{align*}
		& \frac{\mathrm{d}}{\mathrm{d}t} \left[ \frac{1+\theta}{\theta} \|N\|_2^2 + \|W\|_{W_2^1}^2 \right]+ \| N\|_{W_2^1}^2 + 2\delta \|W\|_{W_2^2}^2 \\
		& \qquad \le C \big( 1 + \psi \big) \|N\|_2^2 + C \|n_2\|_{W_2^1} \|W\|_{W_2^2} \|W\|_{W_2^1} + C \big( \|N\|_2^2 + \|W\|_{W_2^1}^2 \big) \\
		& \le \delta  \|W\|_{W_2^2}^2 +  C \big( 1 + \psi \big) \big( \|N\|_2^2 + \|W\|_{W_2^1}^2 \big).
	\end{align*}
	Consequently,
	\begin{equation*}
		\frac{\mathrm{d}}{\mathrm{d}t} \left[ \frac{1+\theta}{\theta} \|N\|_2^2 + \|W\|_{W_2^1}^2 \right] \le C \big( 1 + \psi \big) \left( \frac{1+\theta}{\theta} \|N\|_2^2 + \|W\|_{W_2^1}^2 \right),
	\end{equation*}
	and the integrability~\eqref{eq:c1} of $\psi$, along with Gronwall's lemma, entails that $N=W\equiv 0$ on $(0,T)$ and completes the proof.
\end{proof}

\section*{Acknowledgments}

Part of this work was done while Ph. L. enjoyed the hospitality and support from the Mathematical Institute, Tohoku University, Sendai. We gratefully thank the referees for useful remarks.


%

\bibliographystyle{siam}
\bibliography{SLCSPS}

\end{document}